\documentclass[11pt]{amsart}

\usepackage{graphicx}
\usepackage{mathptmx}

\usepackage{amscd}
\usepackage{amsxtra}
\usepackage{a4wide}
\usepackage{latexsym}
\usepackage{amssymb}
\usepackage{amsfonts}
\usepackage{amsmath}
\usepackage{amsrefs}
\usepackage{mathrsfs}
\usepackage{upref}
\usepackage{txfonts}
\usepackage{color}
\usepackage{tcolorbox}
\usepackage[utf8]{inputenc}

\usepackage[bookmarksnumbered, colorlinks, plainpages]{hyperref}
\hypersetup{colorlinks=true,linkcolor=red, anchorcolor=green,
citecolor=cyan, urlcolor=red, filecolor=magenta, pdftoolbar=true}

\usepackage[left=2cm, right=2cm, top=2cm, bottom=2cm]{geometry}

\allowdisplaybreaks

\newtheorem{theorem}{Theorem}[section]
\newtheorem{lemma}[theorem]{Lemma}

\theoremstyle{definition}

\theoremstyle{remark}
\newtheorem{remark}[theorem]{Remark}
\numberwithin{equation}{section}

\def\GG{\operatorname{G\Gamma}}

\def\esup{\operatornamewithlimits{ess\,sup}}

\def\R{\mathbb R}

\def\ap{\approx}

\def\mf{\mathfrak M}
\def\rad{\operatorname{rad}}

\def\qq{\qquad}
\def\rn{\R^n}

\def\la{\lambda}

\def\vp{\varphi}

\def\rw{\rightarrow}

\def\dn{\downarrow}

\def\R{\mathbb R}

\def\M{\mathfrak M}

\def\mp{{\mathfrak M}}
\def\W{{\mathcal W}}

\begin{document}

\setcounter{page}{1}

\title[Norms of maximal functions between generalized and classical Lorentz spaces]{Norms of maximal functions between generalized \\ and classical Lorentz spaces}

\author[R.Ch. Mustafayev]{RZA MUSTAFAYEV}
\address{RZA MUSTAFAYEV, Department of Mathematics, Kamil \"{O}zda\u{g} Faculty of Science, Karamano\u{g}lu Mehmetbey University, 70200, Karaman, Turkey}
\email{rzamustafayev@gmail.com}

\author[N. B\.{I}LG\.{I}\c{C}L\.{I}]{NEV\.{I}N B\.{I}LG\.{I}\c{C}L\.{I}}
\address{NEV\.{I}N B\.{I}LG\.{I}\c{C}L\.{I}, Republic of Turkey Ministry of National Education, Kirikkale High School, 71100, Kirikkale, Turkey}
\email{nevinbilgicli@gmail.com}

\author[M. Y\i lmaz]{MERVE YILMAZ}
\address{MERVE YILMAZ, Department of Mathematics, Kamil \"{O}zda\u{g} Faculty of Science, Karamano\u{g}lu Mehmetbey University, 70200, Karaman, Turkey}
\email{mervegorgulu@kmu.edu.tr}

\subjclass[2010]{42B25, 42B35}

\keywords{generalized maximal functions, classical and generalized Lorentz spaces,
	iterated Hardy inequalities involving suprema, weights}

\begin{abstract}
	The aim of the paper is to calculate the norm of the generalized maximal operator $M_{\phi,\Lambda^{\alpha}(b)}$,
	defined with $0 < \alpha < \infty$ and functions $b,\,\phi: (0,\infty) \rightarrow (0,\infty)$ for all measurable functions $f$ on ${\mathbb R}^n$ by
	\begin{equation*}
	M_{\phi,\Lambda^{\alpha}(b)}f(x) : = \sup_{Q \ni x} \frac{\|f \chi_Q\|_{\Lambda^{\alpha}(b)}}{\phi (|Q|)}, \qquad x \in \rn,
	\end{equation*}
	from $\GG(p,m,v)$ into $\Lambda^q(w)$. Here $\Lambda^{\alpha}(b)$ and $\GG(p,m,w)$ are the classical and generalized Lorentz spaces, defined as a set of all measurable functions $f$ defined on $\rn$ for which 
	$$
	\|f\|_{\Lambda^{\alpha}(b)} = \bigg( \int_0^{\infty} [f^*(s)]^{\alpha} b(s)\,ds \bigg)^{\frac{1}{\alpha}} < \infty
	\quad \mbox{and} \quad \|f\|_{\GG(p,m,w)} = \bigg( \int_0^{\infty} \bigg( \int_0^x [f^* (\tau)]^p\,d\tau \bigg)^{\frac{m}{p}} v(x)\,dx \bigg)^{\frac{1}{m}} < \infty,
	$$
	respectively.
	
	In order to achieve the goal we reduce the problem to the solution of the inequality
	\begin{equation*}
		\bigg( \int_0^{\infty} \big[ T_{u,b}f^* (x)\big]^q \, w(x)\,dx\bigg)^{\frac{1}{q}} \le C \, \bigg( \int_0^{\infty} \bigg( \int_0^x [f^* (\tau)]^p\,d\tau \bigg)^{\frac{m}{p}} v(x)\,dx \bigg)^{\frac{1}{m}}
	\end{equation*}
	where $w$ and $v$ are weight functions on  $(0,\infty)$. The inequality is required to hold with some positive constant $C$ for all measurable functions defined on measure space $({\mathbb R}^n,dx)$. Here $f^*$ is the non-increasing rearrangement of a measurable function $f$ defined on $\rn$ and $T_{u,b}$ is the iterated Hardy-type operator involving suprema, which is defined for a measurable non-negative function $f$ on $(0,\infty)$ by 
	$$
	(T_{u,b} g)(t) : = \sup_{\tau \in [t,\infty)} \frac{u(\tau)}{B(\tau)} \int_0^{\tau} g(s)b(s)\,ds,\qquad t \in (0,\infty),
	$$ 
	where $u$ and $b$ are two weight functions on $(0,\infty)$ such that $u$ is continuous on $(0,\infty)$ and the function 
	$B(t) : = \int_0^t b(s)\,ds$ satisfies  $0 < B(t) < \infty$ for every $t \in (0,\infty)$.

\end{abstract}

\maketitle


\section{Introduction}\label{in}

Let $({\mathcal R}, \mu)$ be a $\sigma$-finite non-atomic measure space. Denote by
$\mf({\mathcal R})$ the set of all $\mu$-measurable functions on ${\mathcal R}$
and $\mf_0 ({\mathcal R})$ the class of functions in $\mf ({\mathcal R})$ that
are finite $\mu$ - a.e. on ${\mathcal R}$. The symbol $\mp^+ ({\mathcal R})$ stands for the
collection of all $f\in\mp ({\mathcal R})$ which are non-negative on
${\mathcal R}$. 

The non-increasing rearrangement $f^*$ of $f \in \mp_0({\mathcal R})$ is given by
$$
f^* (t) = \inf \big\{ \lambda \ge 0:\, \mu (\{x \in {\mathcal R}:\,|f(x)| > \lambda\}) \le t \big\}, \qquad t \in [0,\mu({\mathcal R})).
$$
The maximal non-increasing rearrangement of $f$ is defined as follows 
\begin{equation}\label{max.func.decr.rear.}
	f^{**}(t) : = \frac{1}{t} \int_0^t f^* (\tau)\,d\tau, \qquad t \in (0,\mu({\mathcal R})).
\end{equation}

Most of the functions which we shall deal with will be defined on ${\mathbb R}^n$ or $(0,\infty)$. In this case, $({\mathcal R}, \mu)$ is ${\mathbb R}^n$ or $(0,\infty)$ endowed with the $n$-dimensional Lebesgue measure or the one-dimensional Lebesgue measure, respectively. We shall write just $\mp^+$ instead of $\mp^+ (0,\infty)$.

Let $\Omega$ be any measurable subset of $\rn$, $n\geq 1$. The family of all weight functions (also called just weights) on $\Omega$, that is, locally integrable non-negative functions on $\Omega$, denoted by $\W(\Omega)$.

For $p\in (0,\infty]$ and $w\in \mp^+(\Omega)$, we define the functional
$\|\cdot\|_{p,w,\Omega}$ on $\mp (\Omega)$ by
\begin{equation*}
	\|f\|_{p,w,\Omega} : = \left\{\begin{array}{cl}
		\bigg(\int_{\Omega} |f(x)|^p w(x)\,dx \bigg)^{1/p} & \qq\mbox{if}\qq p<\infty, \\
		\esup_{\Omega} |f(x)|w(x) & \qq\mbox{if}\qq p=\infty.
	\end{array}
	\right.
\end{equation*}

If, in addition, $w\in \W(\Omega)$, then the weighted Lebesgue space
$L^p(w,\Omega)$ is given by
\begin{equation*}
	L^p(w,\Omega) = \{f\in \mp (\Omega):\,\, \|f\|_{p,w,\Omega} <
	\infty\}
\end{equation*}
and it is equipped with the quasi-norm $\|\cdot\|_{p,w,\Omega}$.

When $w\equiv 1$ on $\Omega$, we write simply $L^p(\Omega)$ and
$\|\cdot\|_{p,\Omega}$ instead of $L^p(w,\Omega)$ and
$\|\cdot\|_{p,w,\Omega}$, respectively.

Quite many familiar function spaces can be defined using the
non-increasing rearrangement of a function. One of the most
important classes of such spaces are the so-called classical Lorentz
spaces.

Let $p \in (0,\infty)$ and $w \in {\mathcal W}(0,\mu({\mathcal R}))$. Then the
classical Lorentz spaces $\Lambda^p (w)$ and $\Gamma^p (w)$ consist
of all functions $f \in {\mathfrak M}({\mathcal R})$ for which
$$
\|f\|_{\Lambda^p(w)} : = \bigg( \int_0^{\mu({\mathcal R})} [f^*(s)]^p w(s)\,ds \bigg)^{\frac{1}{p}} < \infty
\quad \mbox{and} \quad \|f\|_{\Gamma^p(w)} : = \bigg( \int_0^{\mu({\mathcal R})} [f^{**}(s)]^p w(s)\,ds \bigg)^{\frac{1}{p}} < \infty,
$$
respectively. For more information about the Lorentz $\Lambda$ and $\Gamma$ see
e.g. \cite{cpss} and the references therein.

The study of particular problems in the regularity theory of PDE's led to the definition of spaces involving inner integral means involving powers of the non-increasing rearrangements of functions.
The generalized Lorentz $\GG(p,m,v)({\mathcal R,\mu})$ space (denoted simply by $\GG(p,m,v)$), introduced and studied in \cite{FR2008} and \cite{FRZ2009}, is defined as the collection of all $g \in \mp (\mathcal R)$ such that
$$
\|g\|_{\GG(p,m,v)} = \bigg( \int_0^{\mu({\mathcal R})} \bigg( \int_0^x [g^* (\tau)]^p\,d\tau \bigg)^{\frac{m}{p}} v(x)\,dx \bigg)^{\frac{1}{m}} < \infty,
$$
where $m,\,p \in (0,\infty)$, $v \in \W(0,\mu({\mathcal R}))$.

The spaces $\GG(p,m,v)$ cover several types of important function spaces and have plenty of applications. For example, when $\mu({\mathcal R}) = \infty$, $p = 1$, $m > 1$ and $v(t) = t^{-m}w(t)$, $t \in (0,\infty)$, where $w$ is another weight on $(0,\infty)$, then $\GG(p,m,v)$ reduces the spaces $\Gamma^m(w)$. Another important example is obtained when $\mu({\mathcal R}) = 1$, $m = 1$, $p \in (1,\infty)$ and $v(t) = t^{-1}\big( \log ({2} / {t})\big)^{-1/p}$ for $t \in (0,1)$. In this case the space $\GG(p,m,v)$ coincides with the small Lebesgue space, which was originally studied by Fiorenza in \cite{F2000}. In the same paper it was proved that this space is the associate space of the grand Lebesgue space, which is introduced in \cite{IS} in connection with integrability properties of Jacobians. Subsequently, Fiorenza and Karadzhov in \cite{FK} derived an equivalent form of the norm in the small Lebesgue space written in the form of the norm in the $\GG(p,m,v)$ space with the above mentioned parameters and weight. Recently, the connection of the $\GG(p,m,v)$ space with some well-known function spaces have been studied in \cite{AFFGR}. In the present paper we take $(\rn,dx)$ as underlying measure space and use the notation $\GG(p,m,v)$ for $\GG(p,m,v)(\rn,dx)$.  

The study on maximal operators occupies an important place in harmonic analysis. These significant non-linear operators, whose
behavior are very informative in particular in differentiation
theory, provided the understanding and the inspiration for the
development of the general class of singular and potential operators
(see, for instance, \cites{stein1970,guz1975,GR,tor1986,stein1993,graf2008,graf}).

The main example is the Hardy-Littlewood maximal function which is
defined  for locally integrable functions $f$ on $\rn$ by
$$
(Mf)(x) : = \sup_{Q \ni x}\frac{1}{|Q|} \int_Q |f(y)|\,dy, \qquad x \in \rn,
$$
where the supremum is taken over all cubes $Q$ containing $x$. By a cube, we mean an open cube with sides parallel to the coordinate axes.

Another important example is the fractional maximal operator, $M_{\gamma}$, $\gamma \in (0,n)$,
defined  for locally integrable functions $f$ on $\rn$ by
$$
(M_{\gamma} f) (x) := \sup_{Q \ni x} |Q|^{ \gamma / n - 1} \int_{Q}
|f(y)|\,dy,\qquad x \in \rn.
$$

One more example is the fractional maximal operator $M_{s,\gamma,\mathbb A}$ defined in \cite{edop} for all measurable functions $f$ on $\rn$ by
$$
(M_{s,\gamma,\mathbb A} f) (x) : = \sup_{Q \ni x}
\frac{\|f\chi_Q\|_s}{ \|\chi_Q\|_{sn / (n-\gamma),{\mathbb A}}},
\qquad x \in \rn.
$$
Here $s \in (0,\infty)$, $\gamma \in [0,n)$, $\mathbb A = (A_0,A_{\infty}) \in {\mathbb R}^2$ and
$$
\ell^{\mathbb A} (t) : = (1 + |\log t|)^{A_0} \chi_{[0,1]}(t) + (1 +
|\log t|)^{A_{\infty}} \chi_{[1,\infty)}(t),  \qquad t \in (0,\infty).
$$
Recall that the following equivalency holds:
$$
(M_{s,\gamma,\mathbb A} f) (x) \ap \sup_{Q \ni x}
\frac{\|f\chi_Q\|_s}{ |Q|^{(n-\gamma)/(sn)}\ell^{\mathbb A}(|Q|)},
\quad x \in \rn.
$$

Hence, if $s = 1$, $\gamma = 0$ and ${\mathbb A} = (0,0)$, then
$M_{s,\gamma,\mathbb A}$ is equivalent to $M$. If $s = 1$, $\gamma \in
(0,n)$ and ${\mathbb A} = (0,0)$, then $M_{s,\gamma,\mathbb A}$ is
equivalent to $M_{\gamma}$. Moreover, if $s = 1$, $\gamma \in [0,n)$ and ${\mathbb A} \in \R^2$,
then $M_{s,\gamma,\mathbb A} $ is the fractional maximal operator
which corresponds to potentials with logarithmic smoothness treated
in \cites{OT1,OT2}. In particular, if $\gamma = 0$, then
$M_{1,\gamma,\mathbb A}$ is the maximal operator of purely
logarithmic order.

Given $p$ and $q$, $0 < p,\,q < \infty$, let $M_{p,q}$ denote the
maximal operator associated to the Lorentz $L^{p,q}$ spaces defined for all measurable function $f$ on $\rn$ by
$$
M_{p,q} f (x) : = \sup_{Q \ni x} \frac{\|f \chi_Q\|_{p,q}}{\|\chi_Q\|_{p,q}},
$$
where $\|\cdot\|_{p,q}$ is the usual Lorentz norm
$$
\|f\|_{p,q} : = \bigg( \int_0^{\infty} \big[ \tau^{1 / p} f^*
(\tau)\big]^q \frac{d\tau}{\tau}\bigg)^{1 / q}.
$$
This operator was introduced by Stein in \cite{stein1981} in order
to obtain certain endpoint results in differentiation theory. The
operator $M_{p,q}$ have been also considered by other authors, for
instance see  \cite{neug1987,leckneug,basmilruiz,perez1995,ler2005}.

Let $0 < \alpha < \infty$, $b \in \W(0,\infty)$ and $\phi: (0,\infty)
\rightarrow (0,\infty)$. Recall the definition of the generalized maximal function introduced in \cite{musbil} and denoted for all measurable function $f$ on $\rn$ by
\begin{equation}
	M_{\phi,\Lambda^{\alpha}(b)}f(x) : = \sup_{Q \ni x} \frac{\|f \chi_Q\|_{\Lambda^{\alpha}(b)}}{\phi (|Q|)}, \qquad x \in \rn.
\end{equation}

Obviously, $M_{\phi,\Lambda^{\alpha}(b)} = M$, where $M$ is the Hardy-Littlewood maximal operator, when $\alpha = 1$,
$b \equiv 1$ and $\phi (t) = t$ $(t >0)$.
Note that $M_{\phi,\Lambda^{\alpha}(b)} = M_{\gamma}$, where
$M_{\gamma}$ is  the fractional maximal operator, when $\alpha = 1$,
$b \equiv 1$ and $\phi (t) = t^{1 - \gamma / n}$ $(t >0)$ with $0 <
\gamma < n$. Moreover, $M_{\phi,\Lambda^{\alpha}(b)} \ap
M_{s,\gamma,\mathbb A}$, when $\alpha = s$, $b \equiv 1$ and $\phi
(t) = t^{(n - \gamma) / (sn)} \ell^{\mathbb A} (t)$ $(t >0)$ with
$0 < \gamma < n$ and $\mathbb A = (A_0,A_{\infty}) \in {\mathbb
	R}^2$. It is worth also to mention that
$M_{\phi,\Lambda^{\alpha}(b)} = M_{p,q}$, when $\alpha = q$, $b(t) =
t^{q/ p - 1}$ and $\phi (t) = t^{1 / p}$ $(t >0)$.

The boundedness of $M_{\phi,\Lambda^{\alpha}(b)}$ between classical Lorentz spaces $\Lambda$ was completely characterized in \cite{musbil}. The norm of $M_{\phi,\Lambda^{\alpha}(b)}$ between two $\GG$ was calculated in \cite{musbilyil} for wide range of parameters under additional conditions on weight functions. In view of recent increasing interest in generalized Lorentz spaces in our opinion it will be interesting to obtain a characterization of the boundedness of this maximal function between generlized and classical Lorentz spaces.

The iterated Hardy-type operator involving suprema $T_{u,b}$ is defined for non-negative measurable function $g$ on the interval $(0,\infty)$ by 
$$
(T_{u,b} g)(t) : = \sup_{\tau \in [t,\infty)}
\frac{u(\tau)}{B(\tau)} \int_0^{\tau} g(y)b(y)\,dy,\qquad t \in (0,\infty),
$$
where $u$ and $b$ are two weight functions on $(0,\infty)$ such that $u$ is continuous on $(0,\infty)$ and the function $B(t) : = \int_0^t b(s)\,ds$ satisfies  $0 < B(t) < \infty$ for every $t \in (0,\infty)$. 
Such operators have been found indispensable in the search for optimal pairs of rearrangement-invariant norms for which a Sobolev-type inequality holds (cf. \cite{kerp}). They constitute a very useful tool for characterization of the associate norm of an
operator-induced norm, which naturally appears as an optimal domain norm in a Sobolev embedding (cf. \cite{pick2000, pick2002}). Supremum operators are also very useful in limiting interpolation theory as can be seen from their appearance for example in \cite{cwikpys, dok, evop, pys}.

In the present paper it is shown under appropriate conditions on parameters and weight functions that the inequality 
\begin{equation*}
\bigg( \int_0^{\infty} \big[ \big(M_{\phi,\Lambda^{\alpha}(b)}f\big)^* (x)\big]^q  w(x)\,dx\bigg)^{\frac{1}{q}} \le C \bigg( \int_0^{\infty} \bigg( \int_0^x [f^* (\tau)]^{p}\,d\tau \bigg)^{\frac{m}{p}} v(x)\,dx \bigg)^{\frac{1}{m}}
\end{equation*}
holds for all $f \in \mp (\rn)$ if and only if the inequality
$$
\bigg( \int_0^{\infty} \big[ T_{B/\phi^{\alpha},b} h^* (t) \big]^{\frac{q}{\alpha}} w(x)\,dx\bigg)^{\frac{1}{q}} \le C \bigg( \int_0^{\infty} \bigg( \int_0^x [h^* (\tau)]^{\frac{p}{\alpha}}\,d\tau \bigg)^{\frac{m}{p}} v(x)\,dx \bigg)^{\frac{1}{m}}
$$
holds for all $h \in \mp (\rn)$ (see Theorem \ref{main.reduc.thm}).

The above-mentioned observation motivates investigation of the following restricted inequality for $T_{u,b}$:
\begin{equation}\label{main.ineq.}
	\bigg( \int_0^{\infty} \big[ T_{u,b}f^* (x)\big]^q\, w(x)\,dx\bigg)^{\frac{1}{q}} \le C \, \bigg( \int_0^{\infty} \bigg( \int_0^x [f^* (\tau)]^p\,d\tau \bigg)^{\frac{m}{p}} v(x)\,dx \bigg)^{\frac{1}{m}}.
\end{equation}
Here $m,\,p,\,q$ are positive real numbers and $w,\,v$ are weight functions on  $(0,\infty)$. 

The method used for solution of inequalty \eqref{main.ineq.} is based on the combination of the duality techniques with the formula 
\begin{equation*}
	\sup_{g:\,\int_0^x g \le \int_0^x f} \int_0^{\infty} g(x)w(x)\,dx = \int_0^{\infty} f(x) \bigg( \sup_{t \in [x,\infty)} w(t)\bigg)\,dx
\end{equation*}
from \cite{Sinn}, which holds for $f,\,w \in \mp^+(0,\infty)$. On the other hand, it uses estimates of optimal constants in weighted Hardy-type inequalites, as well as in weighted inequalities for a superposition of the supremal or Copson operator with the Hardy operator or the Copson operator. Detailed information on materials that are used in the proofs of the main results is given in the following section.

However, we are not able to solve the inequality under the restrictions $1 < p \le q < m < \infty$ or $1 < q < \min\{p,m\} < \infty$, since for these values of parameters the conditions that characterize the weighted iterated Hardy-type inequalities contains more complicated expressions and the approach used in our paper needs an improvement.

Throughout the paper, we always denote by  $C$ a positive
constant, which is independent of main parameters but it may vary
from line to line. However a constant with subscript such as $C_1$
does not change in different occurrences. By $a\lesssim b$,
we mean that $a\leq \la b$, where $\la >0$ depends on
inessential parameters. If $a\lesssim b$ and $b\lesssim a$, we write
$a\approx b$ and say that $a$ and $b$ are  equivalent. 

As usual, we put $0 \cdot \infty = 0$, $\infty / \infty =0$ and $0/0 = 0$. If $p\in [1,+\infty]$, we define $p'$ by $1/p + 1/p' = 1$.

The paper is organized as follows. We start with formulations of background material in Section \ref{BM}. In Section \ref{s3} we present solution of the restricted inequality. Finally, in Section \ref{BofMF}, we calculate the norm of generalized maximal operator from $\GG$ spaces into $\Lambda$ spaces.


 





\section{Background material}\label{BM}

In this section we collect background material that will be used in the proofs of the main theorems.

We begin with the following characterization of the norm of the associate space of $\GG(p,m,v)$ given in \cite{GPS}. Recall that the associate space $\GG(p,m,v)^{\prime}$ of $\GG(p,m,v)$ is defined as the collection of all functions $g \in \mp (\rn)$ such that
$$
\|g\|_{\GG(p,m,v)^{\prime}} = \sup_{\|f\|_{\GG(p,m,v)} \le 1} \int_0^{\infty} f^* (t) g^* (t)\,dt < \infty.
$$
As it is mentioned in \cite{GPS}, it is reasonable to adopt a general assumption that $p,\,m$ and $v$ are such that
\begin{equation}\label{nontriv}
\int_0^t v(s)s^{\frac{m}{p}}\,ds + \int_t^{\infty} v(s)\,ds < \infty,  \qquad t \in (0,\infty), 
\end{equation}
because if this requirement is not satisfied, then $\GG(p,m,v) = \{0\}$. 

Under the assumption \eqref{nontriv}, we denote
\begin{equation}\label{defof_v}
v_0 (t) : = t^{\frac{m}{p} - 1}\int_0^t v(s)s^{\frac{m}{p}}\,ds \int_t^{\infty} v(s)\,ds, \qquad t \in (0,\infty),
\end{equation}
and
\begin{equation}\label{defof_u}
v_1(t) : = \int_0^t v(s)s^{\frac{m}{p}}\,ds + t^{\frac{m}{p}} \int_t^{\infty} v(s)\,ds, \qquad t \in (0,\infty).
\end{equation}

Moreover, we assume that a weight $v$ is non-degenerate (with respect to the power function $t^{m / p}$), that is,
\begin{equation}\label{nondegen}
\int_0^1 v(s)\,ds = \int_1^{\infty} v(s)s^{\frac{m}{p}}\,ds = \infty.
\end{equation}

We denote the set of all weight functions satisfying conditions \eqref{nontriv} and \eqref{nondegen}  by $\W_{m,p}(0,\infty)$.
\begin{theorem}\label{assosGG}
	Assume that $0 < m,\,p < \infty$ and $v \in \W_{m,p}(0,\infty)$.
	
	{\rm (i)} Let $0 < m \le 1$ and $0 < p \le 1$. Then
	$$
	\|g\|_{\GG(p,m,v)^{\prime}} \ap \sup_{t \in (0,\infty)} g^{**}(t) \frac{t}{v_1(t)^{\frac{1}{m}}};
	$$
	
	{\rm (ii)} Let $0 < m \le 1$ and $1 < p < \infty$. Then
	$$
	\|g\|_{\GG(p,m,v)^{\prime}} \ap \sup_{t \in (0,\infty)} \bigg( \int_t^{\infty} g^{**}(s)^{p'}\,ds \bigg)^{\frac{1}{p'}} \frac{t^{\frac{1}{p}}}{v_1(t)^{\frac{1}{m}}};
	$$
	
	{\rm (iii)} Let $1 < m < \infty$ and $0 < p \le 1$. Then
	$$
	\|g\|_{\GG(p,m,v)^{\prime}} \ap \bigg( \int_0^{\infty} g^{**}(t)^{m'} \frac{t^{m'} v_0(t)}{v_1(t)^{m' + 1}}\,dt \bigg)^{\frac{1}{m'}};
	$$
	
	{\rm (iv)} Let $1 < m < \infty$ and $1 < p < \infty$. Then
	\begin{align*}
		\|g\|_{\GG(p,m,v)^{\prime}} \ap & \,\bigg( \int_0^{\infty} \bigg( \int_t^{\infty} g^{**}(s)^{p'}\,ds \bigg)^{\frac{m'}{p'}} \frac{t^{\frac{m'}{p'}}v_0(t)}{v_1(t)^ {m' + 1}}\,dt \bigg)^{\frac{1}{m'}}.
	\end{align*}	
\end{theorem}

We recall the following well-known duality principle in weighted Lebesgue spaces. 
\begin{theorem}
	Let $p > 1$, $f \in \M^+$	and $w \in \W (0,\infty)$. Then
	$$
	\bigg( \int_0^{\infty} f(t)^p w(t)\,dt \bigg)^{\frac{1}{p}} = \sup_{h \in \M^+ }\frac{\int_0^{\infty} f(t)h(t)\,dt}{\bigg(\int_0^{\infty}h(t)^{p'} w(t)^{1-p'}\,dt\bigg)^{\frac{1}{p'}}}.
	$$
\end{theorem}

We will use the following statement. 
\begin{theorem}\cite[Theorem 2.1]{Sinn}\label{transfermon}
	Suppose $f,\,w \in \mp^+$. Then
	\begin{equation}\label{transmon}
	\sup_{g:\,\int_0^x g \le \int_0^x f} \int_0^{\infty} g(x)w(x)\,dx = \int_0^{\infty} f(x) \bigg( \sup_{t \in [x,\infty)} w(t)\bigg)\,dx.
	\end{equation}
\end{theorem}

Let us now recall (now classical) well-known characterizations of weights for which Hardy and Copson inequalities hold. The following two theorems, which are, incidentally, exactly one hundred years old, are absolutely
indispensable   in various parts of mathematics  (cf. \cite{ok,kp,kufmalpers,kufperssam}).
\begin{theorem}\label{thm.Hardy}
Let $1 < p,\,q < \infty$ and $v,\,w \in \W (0,\infty)$. Denote the best constant in the inequality
\begin{equation*}
\bigg( \int_0^{\infty} \bigg( \int_0^x f(s)\,ds \bigg)^q w(x)\,dx\bigg)^{\frac{1}{q}} \le C \, \bigg( \int_0^{\infty} f(x)^p v(x) \,dx\bigg)^{\frac{1}{p}}, \quad f \in \mp^+,
\end{equation*}
by
$$
C_1 : = \sup_{f \in \mp^+} \frac{\bigg( \int_0^{\infty} \bigg( \int_0^x f(s)\,ds \bigg)^q w(x)\,dx\bigg)^{\frac{1}{q}}}{\bigg( \int_0^{\infty} f(x)^p v(x) \,dx\bigg)^{\frac{1}{p}}}.
$$

{\rm (a)} Let $p \le q$. Then 
$$
C_1 \approx \sup_{t \in (0,\infty)} \bigg( \int_x^{\infty} w(s)\,ds\bigg)^{\frac{1}{q}} \bigg( \int_0^x v(s)^{1-p'}\,ds\bigg)^{\frac{1}{p}}.
$$

{\rm (b)} Let $q < p$. Then 
$$
C_1 \approx \bigg( \int_0^{\infty} \bigg( \int_x^{\infty} w(s)\,ds\bigg)^{\frac{q}{p - q}} \, w(x) \, \bigg( \int_0^x v(s)^{1-p'}\,ds\bigg)^{\frac{q(p-1)}{p-q}} \,dx \bigg)^{\frac{p-q}{pq}}.
$$
\end{theorem}	

\begin{theorem}\label{thm.Copson}
Let $1 < p,\,q < \infty$ and $v,\,w \in \W (0,\infty)$. Denote the best constant in the inequality
\begin{equation*}
\bigg( \int_0^{\infty} \bigg( \int_x^{\infty} f(s)\,ds \bigg)^q w(x)\,dx\bigg)^{\frac{1}{q}} \le C \, \bigg( \int_0^{\infty} f(x)^p v(x) \,dx\bigg)^{\frac{1}{p}}, \quad  f \in \mp^+,
\end{equation*}
by
$$
C_2 : = \sup_{f \in \mp^+} \frac{\bigg( \int_0^{\infty} \bigg( \int_x^{\infty} f(s)\,ds \bigg)^q w(x)\,dx\bigg)^{\frac{1}{q}}}{\bigg( \int_0^{\infty} f(x)^p v(x) \,dx\bigg)^{\frac{1}{p}}}.
$$
	
{\rm (a)} Let $p \le q$. Then 
$$
C_2 \approx \sup_{t \in (0,\infty)} \bigg( \int_0^x w(s)\,ds\bigg)^{\frac{1}{q}} \bigg( \int_x^{\infty} v(s)^{1-p'}\,ds\bigg)^{\frac{1}{p}}.
$$
	
{\rm (b)} Let $q < p$. Then 
$$
C_2 \approx \bigg( \int_0^{\infty} \bigg( \int_0^x w(s)\,ds\bigg)^{\frac{q}{p - q}} \, w(x) \, \bigg( \int_x^{\infty} v(s)^{1-p'}\,ds\bigg)^{\frac{q(p-1)}{p-q}} \,dx \bigg)^{\frac{p-q}{pq}}.
$$
\end{theorem}	

We next quote results concerning characterizations of inequalities involving supremum operators in the following two statements.
\begin{theorem}\label{thm44b}
	Let $ 1 < p < \infty $. Given $t \in [0,\infty)$ assume that $u \in \W (t,\infty) \cap C (t,\infty)$ and  $a, v,\,w \in \W (t,\infty)$ such that $ 0 < \int_{t}^{x} v < \infty $ and $ 0 < \int_{t}^{x} w < \infty $ for $x \in (t,\infty)$. Then the inequality
	\begin{equation}\label{gogopick3}
	\int_t^{\infty} \bigg[\sup_{y \in [x,\infty)} u(y) \int_t^y g(s)\, ds\bigg] w(x)\, dx \le C \, \bigg(\int_t^{\infty} g(x)^p v(x) \, dx\bigg)^{\frac{1}{p}}
	\end{equation}
	holds for all $ g\in \mathfrak{M}^{+} [t,\infty)$ if and only if
	\begin{equation*}
	D_1 := \bigg(\int_{t}^{\infty} \bigg[\sup_{\tau \in [x,\infty)} \bigg[\sup_{y \in [\tau,\infty)}u(y)^{p'}\bigg] \bigg(\int_{t}^{\tau} v(s)^{1 - p'}\,ds\bigg)\bigg] \bigg(\int_{t}^{x} w(s)\, ds\bigg)^{p' - 1} w(x) \,dx \bigg)^{\frac{1}{p'}} < \infty
	\end{equation*}
	and
	\begin{equation*}
	D_2 := \bigg(\int_{t}^{\infty} \bigg(\int_{x}^{\infty} \bigg[\sup_{\tau \in [y,\infty)} u(\tau)\bigg] w(y) \, dy\bigg)^{p' - 1} \bigg[\sup_{\tau \in [x,\infty)} u(\tau)\bigg] \bigg(\int_{t}^{x} v(s)^{1 - p'}\,ds\bigg)\, w(x)\,dx \bigg)^{\frac{1}{p'}} < \infty.
	\end{equation*}
	Moreover, the least constant $C$ such that \eqref{gogopick3} holds for all $ g\in \mathfrak{M}^{+} $ satisfies $ C \approx D_1 + D_2$.
\end{theorem}

\begin{proof}
	The statement was formulated in \cite[Theorem 4.4]{gop} for $t = 0$. The proof directly follows by using change of variables of the type $x + t = y$ several times when $t > 0$.
\end{proof}

\begin{theorem}\label{krepelathm6b} 
	Let $ 1 < p < \infty $. Given $t \in [0,\infty)$ assume that $u \in \W (t,\infty) \cap C (t,\infty)$ and  $a, v,\,w \in \W (t,\infty)$ such that $ 0 < \int_{t}^{x} v < \infty $ and $ 0 < \int_{t}^{x} w < \infty $ for $x \in (t,\infty)$. Then the inequality
	\begin{equation}\label{krepelathm6b+}
	\int_t^{\infty} \bigg[\sup_{y \in [x,\infty)} u(y) \int_{y}^{\infty} g(s)\, ds\bigg]\, w(x)\, dx \le C \, \bigg(\int_t^{\infty} g(x)^p v(x) \, dx\bigg)^{\frac{1}{p}}
	\end{equation}
	holds for all $ g\in \mathfrak{M}^{+} [t,\infty)$ if and only if
	\begin{equation*}
	E_1 := \bigg(\int_t^{\infty} \bigg[\sup_{\tau \in [x,\infty)} u(\tau)^{p'} \bigg( \int_{\tau}^{\infty} v(s)^{1 - p'}\,ds\bigg)\bigg]\bigg(\int_t^x w(s)\, ds\bigg)^{p' - 1} \, w(x)\,dx \bigg)^{\frac{1}{p'}} < \infty
	\end{equation*}
	and
	\begin{equation*}
	E_2 := \bigg(\int_t^{\infty} \bigg(\int_t^x \bigg[\sup_{y\in [s,x]}u(y)\bigg] \, w(s)\, ds\bigg)^{p' - 1}\bigg[ \sup_{\tau \in [x,\infty)} u(\tau) \bigg( \int_{\tau}^{\infty} 	v(s)^{1 - p'}\,ds \bigg) \bigg] \, w(x)\,dx \bigg)^{\frac{1}{p'}} < \infty.
	\end{equation*}
	Moreover, the least constant $C$ such that \eqref{krepelathm6b+} holds for all $g\in \mathfrak{M}^{+}$ satisfies $C \approx E_1 + E_2$.
\end{theorem}

\begin{proof}
The statement was formulated in \cite[Theorem 6]{krep2016} for $t = 0$. The proof directly follows by using change of variables of the type $x + t = y$ several times when $t > 0$.
\end{proof}

Investigation of weighted iterated Hardy-type inequalities started with studying of the inequality
\begin{equation}\label{mainn0}
\bigg( \int_0^{\infty} \bigg(\int_0^t \bigg( \int_s^{\infty} h(y)\,dy\bigg)^m u(s)\,ds \bigg)^{\frac{q}{m}}w(t)\,dt\bigg)^{\frac{1}{q}} \le C \bigg(\int_0^{\infty} h(t)^pv(t)\,dt\bigg)^{\frac{1}{p}}, \qquad h \in \mp^+. 
\end{equation}
Note that inequality \eqref{mainn0} have been considered in the case $m=1$ in \cite{gop2009} (see also \cite{g1}), where the result was
presented without proof, and in the case $p=1$ in \cite{gjop} and \cite{ss}, where the special
type of weight function $v$ was considered. Recall that the inequality has been completely characterized	in \cite{GMP1} and \cite{GMP2} in the case $0<m<\infty$, $0<q\leq \infty$, $1 \le p < \infty$ by using discretization and anti-discretization methods. Another approach to get the characterization of inequality \eqref{mainn0} was presented in \cite{PS_Proc_2013}. The characterization of the inequality can be reduced to the characterization of the weighted Hardy	inequality on the cones of non-increasing functions (see, \cite{gog.mus.2017_1} and \cite{gog.mus.2017_2}). Different approach to solve iterated Hardy-type inequalities has been given in \cite{mus.2017}.

As it was mentioned in \cite{gog.mus.2017_1} the characterization of "dual" inequality
\begin{equation}\label{iterH2}
\bigg( \int_0^{\infty} \bigg(\int_t^{\infty} \bigg( \int_0^s h(y)\,dy\bigg)^m u(s)\,ds \bigg)^{\frac{q}{m}}w(t)\,dt\bigg)^{\frac{1}{q}} \le C \bigg(\int_0^{\infty} h(t)^pv(t)\,dt\bigg)^{\frac{1}{p}}, \qquad h \in \mp^+.
\end{equation}
can be easily obtained from the solutions of inequality \eqref{mainn0}, which was presented in \cite{GKPS}. 

\begin{theorem}\cite[Theorem 2.9, (a) and (c)]{GKPS}\label{gks} 
	Let $p,\,q,\,m \in (1,\infty)$ and $u,\,w,\,v \in \W (0,\infty)$. Assume that the following non-degeneracy conditions are satisfied:
	\begin{itemize}
		\item 
		$u$ is strictly positive, $\int_t^{\infty} u(s)\,ds < \infty$ for all $t \in (0,\infty)$, $\int_0^{\infty} u(s)\,ds = \infty$, 
		
		\item 
		$\int_0^t w(s)\,ds < \infty$, $\int_t^{\infty} w(s) \bigg( \int_s^{\infty} u(y)\,dy\bigg)^{\frac{q}{m}}\,ds < \infty$ for all $t \in (0,\infty)$,
		
		\item 
		$\int_0^1 w(s) \bigg( \int_s^{\infty} u(y)\,dy\bigg)^{\frac{q}{m}}\,ds = \infty$, $\int_1^{\infty} w(s)\,ds = \infty$.
	\end{itemize}
	Let
	$$
	C = \sup_{h \in \mp^+ } \frac{\bigg( \int_0^{\infty} \bigg(\int_t^{\infty} \bigg( \int_0^s h(y)\,dy\bigg)^m u(s)\,ds \bigg)^{\frac{q}{m}}w(t)\,dt\bigg)^{\frac{1}{q}}}{\bigg(\int_0^{\infty} h(t)^pv(t)\,dt\bigg)^{\frac{1}{p}}}
	$$
	
	{\rm (a)} If $p \le \min\{ m,\,q\}$, then $C \approx F_1 + F_2$, where
	$$
	F_1 = \sup_{t \in (0,\infty)} \bigg( \int_0^t w(s) \,ds \bigg)^{\frac{1}{q}}  \, \bigg( \int_t^{\infty} u(s)\,ds \bigg)^{\frac{1}{m}} \bigg(\int_0^t v(\tau)^{1-p'}\,d\tau\bigg)^{\frac{1}{p'}}
	$$
	and
	$$
	F_2 = \sup_{t \in (0,\infty)} \bigg( \int_t^{\infty} \bigg( \int_s^{\infty} u(y)\,dy \bigg)^{\frac{q}{m}} w(s) \, ds \bigg)^{\frac{1}{q}} \bigg(\int_0^t v(s)^{1-p'}\,ds\bigg)^{\frac{1}{p'}}.
	$$
	
	{\rm (b)} If $m < p \le q$, then $C \approx F_2 + F_3$, where
	$$
	F_3 = \sup_{t \in (0,\infty)} \bigg( \int_0^t w(s) \,ds \bigg)^{\frac{1}{q}} \bigg(\int_t^{\infty} \bigg( \int_s^{\infty} u(y)\,dy \bigg)^{\frac{p}{p - m}} \bigg( \int_0^s v(\tau)^{1 - p'}\,d\tau\bigg)^{\frac{p(m-1)}{p - m}} v(s)^{1 - p'}\,ds\bigg)^{\frac{p - m}{pm}}.
	$$
\end{theorem}

Another pair of "dual" weighted iterated Hardy-type inequalities are
\begin{equation}\label{iterH1}
\bigg( \int_0^{\infty} \bigg(\int_t^{\infty} \bigg( \int_s^{\infty} h(y)\,dy\bigg)^m u(s)\,ds \bigg)^{\frac{q}{m}}w(t)\,dt\bigg)^{\frac{1}{q}} \le C \bigg(\int_0^{\infty} h(t)^pv(t)\,dt\bigg)^{\frac{1}{p}}, \qquad h \in \mp^+ 
\end{equation}
and 
\begin{equation}\label{iterH3}
\bigg( \int_0^{\infty} \bigg(\int_0^t \bigg( \int_0^s h(y)\,dy\bigg)^m u(s)\,ds \bigg)^{\frac{q}{m}}w(t)\,dt\bigg)^{\frac{1}{q}} \le C \bigg(\int_0^{\infty} h(t)^pv(t)\,dt\bigg)^{\frac{1}{p}}, \qquad h \in \mp^+.
\end{equation}
Both of them were characterized in \cite{gog.mus.2017_1} by so-called "flipped" conditions. The "classical" conditions ensuring the validity of \eqref{iterH1} was recently presented in \cite{krepick}.

\begin{theorem}\cite[Theorem 1.1, (a) and (c)]{krepick}\label{krepick} 
	Let $p,\,q,\,m \in (1,\infty)$ and $u,\,w,\,v$ be weights such that the pair $(u,w)$ is admissible with respect to $(m.q)$, that is,
	$$
	0 < \int_0^t \bigg(\int_s^t u(y)\,dy\bigg)^{\frac{q}{m}} w(s)\,ds < \infty, \qquad t \in (0,\infty). 
	$$
	Let
	$$
	C = \sup_{h \in \mp^+ } \frac{\bigg( \int_0^{\infty} \bigg(\int_t^{\infty} \bigg( \int_s^{\infty} h(y)\,dy\bigg)^m u(s)\,ds \bigg)^{\frac{q}{m}}w(t)\,dt\bigg)^{\frac{1}{q}}}{\bigg(\int_0^{\infty} h(t)^pv(t)\,dt\bigg)^{\frac{1}{p}}}
	$$
	
	{\rm (a)} If $p \le \min\{m,\,q\}$, then $C \approx G_1$, where
	$$
	G_1 = \sup_{t \in (0,\infty)} \bigg( \int_0^t w(s) \, \bigg( \int_s^t u(y)\,dy \bigg)^{\frac{q}{m}} ds \bigg)^{\frac{1}{q}} \bigg(\int_t^{\infty} v(s)^{1-p'}\,ds\bigg)^{\frac{1}{p'}}.
	$$
	
	{\rm (b)} If $m < p \le q$, then $C \approx G_1 + G_2$, where
	$$
	G_2 = \sup_{t \in (0,\infty)} \bigg( \int_0^t w(s) \,ds \bigg)^{\frac{1}{q}} \bigg(\int_t^{\infty} \bigg( \int_t^s u(y)\,dy \bigg)^{\frac{m}{p - m}} u(s) \bigg( \int_s^{\infty} v(\tau)^{1 - p'}\,d\tau\bigg)^{\frac{m(p-1)}{p - m}}\,ds\bigg)^{\frac{p - m}{pm}}.
	$$
\end{theorem}

We will apply the following "gluing" lemma.
\begin{lemma}\cite[Lemma 2.7]{gogmusunv}\label{gluing.lem.0} 
	Let $\alpha$ and $\beta$ be positive numbers. Suppose that $g,\,h \in \M^+ $ and $a \in \W (0,\infty)$ is non-decreasing. Then
	\begin{align*}
		\esup_{x \in (0,\infty)} \bigg( \int_0^{\infty} \bigg(\frac{a(x)}{a(x) + a(t)}\bigg)^{\beta} g(t)\,dt \bigg)^{\frac{1}{\beta}} \bigg(\int_0^{\infty} \bigg(\frac{a(t)}{a(x) + a(t)}\bigg)^{\alpha} h(t)\,dt\bigg)^{\frac{1}{\alpha}} & \\
		& \hspace{-7cm} \ap \esup_{x \in (0,\infty)} \bigg( \int_0^x g(t)\,dt\bigg)^{\frac{1}{\beta}} \bigg(\int_x^{\infty} h(t)\,dt\bigg)^{\frac{1}{\alpha}} + \esup_{x \in (0,\infty)} \bigg( \int_x^{\infty} a(t)^{-\beta} g(t)\,dt \bigg)^{\frac{1}{\beta}}  \bigg( \int_0^x a(t)^{\alpha}h(t)\,dt\bigg)^{\frac{1}{\alpha}}.
	\end{align*}
\end{lemma}
We recall the following "an integration by parts" formula. 
\begin{theorem}\cite[Theorem 2.1]{musbil_2}\label{thm.IBP.0} 
	Let $\alpha > 0$. Let $g$ be a non-negative function on $(0,\infty)$ such that $0 < \int_0^t g < \infty$ for all $t \in (0,\infty)$ and let $f$ be a non-negative non-increasing right-continuous function on $(0,\infty)$. Then
	\begin{align*}
		A_1 : = \int_0^{\infty} \bigg( \int_0^t g \bigg)^{\alpha} g(t) [f(t) - \lim_{t \rw +\infty} f(t)]\,dt < \infty \quad \Longleftrightarrow \quad A_2 : = \int_{(0,\infty)} \bigg( \int_0^t g \bigg)^{\alpha + 1}\,d[-f(t)] < \infty.
	\end{align*}
	Moreover, $A_1 \approx A_2$.
\end{theorem}

We are going to make use of the following remark in order to shorten some calculations in the proofs.
\begin{remark}\label{rem}
	Let $w \in \W (0,\infty)$ and $F$ be any non-negative continuous function on $(0,\infty)$. 
	
	Since 
	\begin{equation*}
	\sup_{x \le \tau} F(\tau) \chi_{(0,t]}(\tau)  = \left\{\begin{array}{cl}
	\sup_{x \le \tau \le t} F(\tau) & \quad \mbox{for}\quad x \le t, \\
	0 & \quad \mbox{for}\quad t < x,
	\end{array}
	\right.
	\end{equation*}
	for any $0 < x,\,t < \infty$, then
	\begin{align*}
	\int_0^{\infty} \bigg( \sup_{x \le \tau} F(\tau) \chi_{(0,t]}(\tau) \bigg) w(x)\,dx \approx \int_0^{t} \bigg( \sup_{x \le \tau \le t}F(\tau) \bigg) \, w(x)\,dx
	\end{align*}
	holds for $0 < t < \infty$.
	
	Similarly, since 
	\begin{equation*}
	\sup_{x \le \tau} F(\tau) \chi_{[t, \infty)}(\tau)  = \left\{\begin{array}{cl}
	\sup_{t \le \tau} F(\tau) & \quad \mbox{for}\quad x \le t, \\
	\sup_{x \le \tau} F(\tau) & \quad \mbox{for}\quad t < x,
	\end{array}\right.
	\end{equation*}
	for any $0 < x,\,t < \infty$, then
	\begin{align*}
	\int_0^{\infty} \bigg( \sup_{x \le \tau} F(\tau) \chi_{[t,\infty)}(\tau) \bigg) w(x)\,dx & \\ 
	& \hspace{-4cm} \approx	\bigg( \sup_{t \le \tau}  F(\tau) \bigg) \int_0^{t} \, w(x)\, dx + \int_t^{\infty} \bigg( \sup_{x \le \tau} F(\tau) \bigg) \, w(x)\,dx 
	\end{align*}
	holds for $0 < t < \infty$.	
\end{remark}


\section{Characterizations of restricted inequalities for $T_{u,b}$}\label{s3}

We start this section with some historical remarks concerning restricted inequalities related to the operator $T_{u,b}$. 

The notation $\mp^+ ((0,\infty);\dn)$ is used to denote the subset of those functions from $\mf^+ (0,\infty)$ which are non-increasing on $(0,\infty)$.

Recall that the inequality 
\begin{equation}\label{Tub.thm.1.eq.1}
\|T_{u,b}f \|_{q,w,(0,\infty)} \le C \| f \|_{p,v,(0,\infty)}, \qquad f \in \mp^+ ((0,\infty);\dn)
\end{equation}
was characterized in \cite[Theorem 3.5]{gop} under condition
$$
\sup_{t \in (0,\infty)} \frac{u(t)}{B(t)} \int_0^t
\frac{b(\tau)}{u(\tau)}\,d\tau < \infty.
$$
However, the case when $0 < p \le 1 < q < \infty$ was not considered in \cite{gop}. It is also worth to mention that in the case when $1 < p < \infty$, $0 < q < p < \infty$, $q \neq 1$ \cite[Theorem 3.5]{gop} contains only discrete condition. In
\cite{gogpick2007} the new reduction theorem was obtained when $0 < p \le 1$, and this technique allowed to characterize inequality \eqref{Tub.thm.1.eq.1} when $b \equiv 1$, and in the case when $0 < q< p \le 1$, \cite{gogpick2007} contains only discrete condition. The complete characterizations of inequality  \eqref{Tub.thm.1.eq.1} for $0 < q \le \infty$, $0 < p \le \infty$ were given in \cite{GogMusISI} and \cite{musbil}. Using the results in  \cites{PS_Proc_2013,PS_Dokl_2013,PS_Dokl_2014,P_Dokl_2015}, another characterization of  \eqref{Tub.thm.1.eq.1}  was obtained  in  \cite{StepSham} and \cite{Sham}. 

Note that the inequality
\begin{equation*}
\bigg( \int_0^{\infty} \bigg( \int_0^x T_{u,b} f (t)\,dt\bigg)^q w(x)\,dx\bigg)^{\frac{1}{q}} \le C \, \bigg( \int_0^{\infty} f(x)^p v(x) \,dx\bigg)^{\frac{1}{p}}, \qquad f \in \mp^+ ((0,\infty);\dn) 
\end{equation*}
was characterized in \cite[Theorem 6.1]{musbil_2} for $1 < p,\,q < \infty$, where $w$ and $v$ are weight functions on  $(0,\infty)$.

Recall that the inequality
\begin{equation}\label{eq.0.1}
\bigg( \int_0^{\infty} \bigg( \int_0^x \big[ T_{u,b}f^* (t)\big]^r\,dt\bigg)^{\frac{q}{r}} w(x)\,dx\bigg)^{\frac{1}{q}} \le C \, \bigg( \int_0^{\infty} \bigg( \int_0^x [f^* (\tau)]^p\,d\tau \bigg)^{\frac{m}{p}} v(x)\,dx \bigg)^{\frac{1}{m}}, \qquad f \in \mp (\rn)
\end{equation}
was investigated in \cite[Theorems 3.3 and 3.4]{musbilyil} for $1 < m < p \le r < q < \infty$ or $1 < m \le r < \min\{p,q\} < \infty$, where $w$ and $v$ are appropriate weight functions on $(0,\infty)$.

In this section we give characterization of the inequality
\begin{equation}\label{eq.1.1}
\bigg( \int_0^{\infty}  \big[ T_{u,b}f^* (x)\big]^q \, w(x)\,dx\bigg)^{\frac{1}{q}} \le C \, \bigg( \int_0^{\infty} \bigg( \int_0^x [f^* (\tau)]^p\,d\tau \bigg)^{\frac{m}{p}} v(x)\,dx \bigg)^{\frac{1}{m}}, \qquad f \in \mp (\rn).
\end{equation}
As it was mentioned in \cite{musbilyil}, inequality \eqref{eq.0.1} for $r = q$ is a special case of inequality \eqref{eq.1.1}.

Denote the best constant in inequality \eqref{eq.1.1} by $K$, that is, 
$$
K : = \sup_{f \in \mp (\rn)} \frac{\bigg( \int_0^{\infty}\big[ T_{u,b}f^* (x)\big]^q w(x)\,dx\bigg)^{\frac{1}{q}}}
{\bigg( \int_0^{\infty} \bigg( \int_0^x [f^* (\tau)]^p\,d\tau \bigg)^{\frac{m}{p}} v(x)\,dx \bigg)^{\frac{1}{m}}} \equiv \sup_{f \in \mp (\rn)} \frac{\bigg( \int_0^{\infty}\big[ T_{u,b}f^* (x)\big]^q w(x)\,dx\bigg)^{\frac{1}{q}}}
{\|f\|_{\GG(p,m,v)}}.
$$	

\begin{lemma} \label{redlemma}
Let $0 < p < \infty$, $0 < m < \infty$, $1 < q < \infty$ and $b \in \W(0,\infty)$ be such that the function 
$B(t)$ satisfies  $0 < B(t) < \infty$ for every $t \in (0,\infty)$. Assume that $u \in \W(0,\infty) \cap C(0,\infty)$ and $v,\,w \in \W(0,\infty)$. 
Then we have
$$ 
K = \sup_{h:\, \int_0^x h \le \int_0^x w} \sup_{\vp \in \mathfrak{M}^{+}} \frac{1}{\|\vp\|_{q',h^{1-q'},(0,\infty)}}
\sup_{f \in \mp (\rn)} \frac{\int_0^{\infty} f^* (y) b(y) \int_y^{\infty} \vp(x) \frac{u(x)}{B(x)} \,dx \,dy}{\|f\|_{\GG(p,m,v)}}.
$$
\end{lemma}

\begin{proof}
Applying Theorem \ref{transfermon}, we have that
\begin{align*}
K = & \sup_{f \in \mp (\rn)} \frac{\bigg( \int_0^{\infty} \sup_{\tau \in [x,\infty)} \bigg[ \frac{u(\tau)}{B(\tau)} \int_0^{\tau} f^* (y) b(y) \,dy \bigg]^q w(x)\,dx\, \bigg)^{\frac{1}{q}}}{\|f\|_{\GG(p,m,v)}} \\ 
= & \sup_{f \in \mp (\rn)} \frac{1}{\|f\|_{\GG(p,m,v)}} \sup_{h:\, \int_0^x h \le \int_0^x w} \bigg(\int_{0}^{\infty} h(x) \bigg(\frac{u(x)}{B(x)} \int_0^{x} f^* (y) b(y)\,dy \bigg)^q \,dx \bigg)^{\frac{1}{q}}
\end{align*}

By duality and Fubini theorem, we get that
\begin{align*}
K = &\sup_{f \in \mp (\rn)} \frac{1}{\|f\|_{\GG(p,m,v)}} \sup_{h:\, \int_0^x h \le \int_0^x w}\sup_{\vp \in \mathfrak{M}^{+}} \frac{\int_0^{\infty} \vp(x) \frac{u(x)}{B(x)} \bigg(\int_0^x f^* (y) b(y)\,dy\bigg) \,dx}{\|\vp\|_{q',h^{1-q'},(0,\infty)}} \\
= & \sup_{f \in \mp (\rn)} \frac{1}{\|f\|_{\GG(p,m,v)}} \sup_{h:\, \int_0^x h \le \int_0^x w} \sup_{\vp \in \mathfrak{M}^{+}}
\frac{\int_0^{\infty} f^* (y) b(y) \int_y^{\infty} \vp(x) \frac{u(x)}{B(x)} \,dx \,dy}{\|\vp\|_{q',h^{1-q'},(0,\infty)}} 
\end{align*}

Interchanging the suprema yields that 
$$ 
K = \sup_{h:\, \int_0^x h \le \int_0^x w} \sup_{\vp \in \mathfrak{M}^{+}} \frac{1}{\|\vp\|_{q',h^{1-q'},(0,\infty)}}
\sup_{f \in \mp (\rn)} \frac{\int_0^{\infty} f^* (y) b(y) \int_y^{\infty} \vp(x) \frac{u(x)}{B(x)} \,dx \,dy}{\|f\|_{\GG(p,m,v)}}.
$$

This completes the proof.	
\end{proof}

\begin{theorem} \label{1stresult}
	Let	$0 < m \le 1, 0 < p \le 1, 1 < q < \infty$ and $b \in \W (0,\infty) \cap \mp^+ ((0,\infty);\dn)$ be such that the function $B(t)$ satisfies  $0 < B(t) < \infty$ for every $t \in (0,\infty)$. Suppose that $u \in \W(0,\infty) \cap C(0,\infty)$, $v \in \W_{m,p}(0,\infty)$ and $w \in \W(0,\infty)$.
	Then
	\begin{align*}
		K \approx & \, \sup_{t \in (0,\infty)} \frac{1}{v_1(t)^{\frac{1}{m}}} \bigg(\int_0^{t}\bigg( \sup_{\tau \in [x,t]} u(\tau)^q \bigg) \, w(x)\,dx\bigg) ^{\frac{1}{q}}
		\\
		& + \, \sup_{t \in (0,\infty)} \frac{B(t)}{v_1(t)^{\frac{1}{m}}} \bigg( \sup_{\tau \in [t,\infty)} \frac{u(\tau)}{B(\tau)} \bigg) \bigg( \int_0^{t}  w(x) \, dx \bigg) ^{\frac{1}{q}}
		\\
		& + \sup_{t \in (0,\infty)} \frac{B(t)}{v_1(t)^{\frac{1}{m}}}\bigg( \int_t^{\infty} \bigg( \sup_{\tau \in [x,\infty)} \bigg( \frac{u(\tau)}{B(\tau)}\bigg)^q \bigg) \, w(x) \, dx \bigg) ^{\frac{1}{q}}.
	\end{align*}
\end{theorem}	

\begin{proof}
	By Lemma \ref{redlemma}, Theorem \ref{assosGG}, (i), and Fubini theorem, we have that
	\begin{align*}
		K \approx & \sup_{h:\, \int_0^x h \le \int_0^x w}  \sup_{\vp \in \mathfrak{M}^{+}} \frac{1}{\|\vp\|_{q',h^{1-q'},(0,\infty)}}
		\sup_{t \in (0,\infty)} \frac{\int_0^{t} b(y) \int_y^{\infty} \vp(x) \frac{u(x)}{B(x)} \,dx \,dy}{v_1(t)^{\frac{1}{m}}} \\
		\approx & \sup_{h:\, \int_0^x h \le \int_0^x w} \sup_{\vp \in \mathfrak{M}^{+}} \frac{1}{\|\vp\|_{q',h^{1-q'},(0,\infty)}}
		\sup_{t \in (0,\infty)} \frac{\int_0^{t} \vp(x) u(x) \, dx}{v_1(t)^{\frac{1}{m}}} 
		\\
		& + \sup_{h:\, \int_0^x h \le \int_0^x w} \sup_{\vp \in \mathfrak{M}^{+}} \frac{1}{\|\vp\|_{q',h^{1-q'},(0,\infty)}}
		\sup_{t \in (0,\infty)} \frac{B(t)\int_t^{\infty} \vp(x) \frac{u(x)}{B(x)} \,dx}{v_1(t)^{\frac{1}{m}}}.
	\end{align*} 

	Interchanging suprema, by duality, we get that
	\begin{align*}
		K \approx & \, \sup_{h:\, \int_0^x h \le \int_0^x w} \sup_{t \in (0,\infty)} \frac{1}{v_1(t)^{\frac{1}{m}}}
		\sup_{\vp \in \mathfrak{M}^{+}} \frac{\int_0^{\infty} \vp(x) u(x) \chi_{(0,t]}(x)\,dx}{\|\vp\|_{q',h^{1-q'},(0,\infty)}} 
		\\
		& + \sup_{h:\, \int_0^x h \le \int_0^x w} \sup_{t \in (0,\infty)} \frac{B(t)}{v_1(t)^{\frac{1}{m}}}
		\sup_{\vp \in \mathfrak{M}^{+}} \frac{\int_0^{\infty} \vp(x) \frac{u(x)}{B(x)} \chi_{[t,\infty)}(x)\,dx}{\|\vp\|_{q',h^{1-q'},(0,\infty)}}
		\\
		\approx & \,\sup_{t \in (0,\infty)} \frac{1}{v_1(t)^{\frac{1}{m}}} \sup_{h:\, \int_0^x h \le \int_0^x w}  \bigg(\int_0^{\infty} h(x) u(x)^q \chi_{(0,t]}(x)\,dx\bigg)^{\frac{1}{q}}
		\\
		& + \sup_{t \in (0,\infty)} \frac{B(t)}{v_1(t)^{\frac{1}{m}}} \sup_{h:\, \int_0^x h \le \int_0^x w} \bigg( \int_0^{\infty} h(x) \bigg( \frac{u(x)}{B(x)}\bigg)^q \chi_{[t,\infty)}(x)\,dx\bigg)^{\frac{1}{q}}.
	\end{align*} 

	Applying Theorem \ref{transfermon}, by Remark \ref{rem}, we arrive at
	\begin{align*}
		K \approx  & \,  \sup_{t \in (0,\infty)} \frac{1}{v_1(t)^{\frac{1}{m}}} \bigg(\int_0^{\infty} \bigg( \sup_{\tau \in [x,\infty)} u(\tau)^q \chi_{(0,t]}(\tau)\bigg) \, w(x)\,dx \bigg)^{\frac{1}{q}}
		\\
		& + \sup_{t \in (0,\infty)} \frac{B(t)}{v_1(t)^{\frac{1}{m}}} \bigg(\int_0^{\infty} \bigg( \sup_{\tau \in [x,\infty)} \bigg( \frac{u(\tau)}{B(\tau)}\bigg)^q \chi_{[t,\infty)}(\tau) \bigg) \, w(x) \, dx \bigg)^{\frac{1}{q}} \\
		\approx & \, \sup_{t \in (0,\infty)} \frac{1}{v_1(t)^{\frac{1}{m}}} \bigg(\int_0^{t}\bigg( \sup_{\tau \in [x,t]} u(\tau)^q \bigg) \, w(x)\,dx\bigg) ^{\frac{1}{q}}
		\\
		& + \, \sup_{t \in (0,\infty)} \frac{B(t)}{v_1(t)^{\frac{1}{m}}} \bigg( \sup_{\tau \in [t,\infty)} \frac{u(\tau)}{B(\tau)} \bigg) \bigg( \int_0^{t}  w(x) \, dx \bigg) ^{\frac{1}{q}}
		\\
		& + \sup_{t \in (0,\infty)} \frac{B(t)}{v_1(t)^{\frac{1}{m}}}\bigg( \int_t^{\infty} \bigg( \sup_{\tau \in [x,\infty)}  \bigg( \frac{u(\tau)}{B(\tau)}\bigg)^q \bigg)  \, w(x) \, dx \bigg) ^{\frac{1}{q}}.
	\end{align*}

	The proof is completed.
\end{proof}	

\begin{theorem}\label{2stresult}
	Let $0 < m \le 1 , 1 < p < \infty, 1 < q < \infty$ and $b \in \W (0,\infty) \cap \mp^+ ((0,\infty);\dn)$ be such that the function $B(t)$ satisfies  $0 < B(t) < \infty$ for every $t \in (0,\infty)$. Suppose that $u \in \W(0,\infty) \cap C(0,\infty)$, $v \in \W_{m,p}(0,\infty)$ and $w \in \W(0,\infty)$.
	\begin{itemize}
		\item[i)] If $p \le q$, then
		\begin{align*}
			K \approx & \sup_{t \in (0,\infty)} \frac{t^{\frac{1}{p}}}{v_1(t)^{\frac{1}{m}}}  \sup_{s \in [t,\infty)} s^{-\frac{1}{p}} \bigg(\int_0^{s}\bigg( \sup_{\tau \in [x,s]} u(\tau)^q \bigg) \, w(x)\,dx\bigg) ^{\frac{1}{q}}
			\\
			& + \sup_{t \in (0,\infty)} \frac{t^{\frac{1}{p}}}{v_1(t)^{\frac{1}{m}}} \sup_{s \in [t,\infty)} \bigg( \int_t^s \bigg(\frac{B(\tau)}{\tau}\bigg)^{p'}\,d\tau\bigg)^{\frac{1}{p'}} \bigg( \sup_{\tau \in [s,\infty)} \frac{u(\tau)}{B(\tau)} \bigg) \bigg( \int_0^{s}  w(x) \, dx \bigg) ^{\frac{1}{q}}
			\\
			& + \sup_{t \in (0,\infty)} \frac{t^{\frac{1}{p}}}{v_1(t)^{\frac{1}{m}}} \sup_{s \in [t,\infty)} \bigg( \int_t^s \bigg(\frac{B(\tau)}{\tau}\bigg)^{p'}\,d\tau\bigg)^{\frac{1}{p'}} \bigg( \int_s^{\infty} \bigg( \sup_{\tau \in [x,\infty)} \bigg( \frac{u(\tau)}{B(\tau)}\bigg)^q \bigg) \, w(x) \, dx \bigg) ^{\frac{1}{q}};
		\end{align*}
				
		\item[ii)] If $q <  p$, then
		\begin{align*}
			K \approx & \sup_{t \in (0,\infty)} \frac{1}{v_1(t)^{\frac{1}{m}}} \bigg( \int_0^{t} \bigg( \sup_{\tau \in [x,t]} u(\tau)^q \bigg) \, w(x) \,dx\bigg)^{\frac{1}{q}} 
			\\
			& + \sup_{t \in (0,\infty)} \frac{t^{\frac{1}{p}}}{v_1(t)^{\frac{1}{m}}} \bigg(\int_0^{t} \, w(x) \,dx\bigg)^{\frac{1}{q}} \bigg( \sup_{\tau \in [t,\infty)} u(\tau) \tau^{-\frac{1}{p}} \bigg)
			\\
			& + \sup_{t \in (0,\infty)} \frac{t^{\frac{1}{p}}}{v_1(t)^{\frac{1}{m}}}  \bigg( \int_t^{\infty} \bigg(\sup_{\tau \in [s,\infty)} u(\tau)^{\frac{pq}{p-q}} \tau^{\frac{q}{q-p}}\bigg) \bigg(\int_{t}^{s} w(x) \, dx \bigg)^{\frac{q}{p-q}} w(s) \,ds\bigg)^{\frac{p-q}{pq}}
			\\
			& + \sup_{t \in (0,\infty)} \frac{t^{\frac{1}{p}}}{v_1(t)^{\frac{1}{m}}}  \bigg(\int_t^{\infty} \bigg(\int_{t}^{s} \bigg( \sup_{y\in [x,s]} u(y)^q \bigg) w(x) \,dx\bigg)^{\frac{q}{p-q}} \bigg(\sup_{\tau \in [s,\infty)} u(\tau)^q \tau^{\frac{q}{q-p}}\bigg)w(s)\, ds\bigg)^{\frac{p-q}{pq}}
			\\
			& + \sup_{t \in (0,\infty)} \frac{t^{\frac{1}{p}}}{v_1(t)^{\frac{1}{m}}} \bigg(\int_0^{t} w(x)\,dx\bigg)^{\frac{1}{q}} \bigg( \sup_{\tau \in [t,\infty)} \frac{u(\tau)}{B(\tau)} \bigg( \int_{t}^{\tau} {\mathcal B}(t,s) \,ds \bigg)^{\frac{p-q}{pq}} \bigg)
			\\
			& + \sup_{t \in (0,\infty)} \frac{t^{\frac{1}{p}}}{v_1(t)^{\frac{1}{m}}}  \bigg( \int_t^{\infty} \bigg[\sup_{\tau \in [x,\infty)}\bigg[\sup_{y \in [\tau,\infty)} \bigg(\frac{u(y)}{B(y)}\bigg)^\frac{pq}{p-q}\bigg]\bigg(\int_{t}^{\tau} \mathcal{B}(t,s) \,ds \bigg) \bigg]\bigg(\int_{t}^{x} w(y)\,dy \bigg)^{\frac{q}{p-q}} w(x) \,dx \bigg)^{\frac{p-q}{pq}}
			\\
			& + \sup_{t \in (0,\infty)} \frac{t^{\frac{1}{p}}}{v_1(t)^{\frac{1}{m}}}  \bigg( \int_t^{\infty}\bigg(\int_{x}^{\infty} \bigg[\sup_{\tau \in [y,\infty)}\bigg(\frac{u(\tau)}{B(\tau)}\bigg)^q\bigg]w(y)\, dy\bigg)^{\frac{q}{p-q}} \bigg[\sup_{\tau \in [x,\infty)}\bigg(\frac{u(\tau)}{B(\tau)}\bigg)^q\bigg]\bigg(\int_{t}^{x}\mathcal{B}(t,s) \,ds \bigg) w(x) \,dx\bigg)^{\frac{p-q}{pq}},
		\end{align*} 
		where 
		$$
		{\mathcal B}(t,s) := \bigg( \int_t^s \bigg( \frac{B(y)}{y} \bigg)^{p'}\,dy \bigg)^{\frac{p(q-1)}{p-q}}  \,  \bigg( \frac{B(s)}{s} \bigg)^{p'}, \qquad 0 < t < s < \infty.
		$$
	\end{itemize}
\end{theorem}	
\begin{proof}
	By  Lemma \ref{redlemma} and Theorem \ref{assosGG}, (ii), we have
	\begin{align*}
		K \approx & \sup_{h:\, \int_0^x h \le \int_0^x w} \sup_{\vp \in \mathfrak{M}^{+}} \frac{1}{\|\vp\|_{q',h^{1-q'},(0,\infty)}}
		\sup_{t \in (0,\infty)} \bigg( \int_t^{\infty} \bigg( \frac{1}{s} \int_0^s b(y) \int_y^{\infty} \vp(x) \frac{u(x)}{B(x)} \,dx \bigg)^{p'} \,ds \bigg)^{\frac{1}{p'}} \frac{t^{\frac{1}{p}}}{v_1(t)^{\frac{1}{m}}} \\
		 \approx & \sup_{h:\, \int_0^x h \le \int_0^x w} \sup_{\vp \in \mathfrak{M}^{+}} \frac{1}{\|\vp\|_{q',h^{1-q'},(0,\infty)}}
		\sup_{t \in (0,\infty)} \frac{t^{\frac{1}{p}}}{v_1(t)^{\frac{1}{m}}} \bigg( \int_t^{\infty} \bigg( \frac{1}{s} \int_0^s \vp(x) u(x) \,dx \bigg)^{p'} \,ds\bigg)^{\frac{1}{p'}} \\
		& + \sup_{h:\, \int_0^x h \le \int_0^x w} \sup_{\vp \in \mathfrak{M}^{+}} \frac{1}{\|\vp\|_{q',h^{1-q'},(0,\infty)}}
		\sup_{t \in (0,\infty)} \frac{t^{\frac{1}{p}}}{v_1(t)^{\frac{1}{m}}} \bigg( \int_t^{\infty} \bigg( \frac{B(s)}{s}  \int_s^{\infty} \vp(x) \frac{u(x)}{B(x)} \,dx \bigg)^{p'} \,ds\bigg)^{\frac{1}{p'}}.
	\end{align*} 
	
	Interchanging the suprema gives
	\begin{align*}
		K \approx & \sup_{h:\, \int_0^x h \le \int_0^x w} \sup_{t \in (0,\infty)} \frac{t^{\frac{1}{p}}}{v_1(t)^{\frac{1}{m}}} \sup_{\vp \in \mathfrak{M}^{+}} \frac{ \bigg( \int_0^{\infty} \bigg( \frac{1}{s} \int_0^s \vp(x) u(x) \,dx \bigg)^{p'} \chi_{[t,\infty)}(s)\,ds\bigg)^{\frac{1}{p'}}}{\|\vp\|_{q',h^{1-q'},(0,\infty)}} 
		\\
		& + \sup_{h:\, \int_0^x h \le \int_0^x w} \sup_{t \in (0,\infty)} \frac{t^{\frac{1}{p}}}{v_1(t)^{\frac{1}{m}}} \sup_{\vp \in \mathfrak{M}^{+}} \frac{\bigg( \int_0^{\infty} \bigg( \frac{B(s)}{s}  \int_s^{\infty} \vp(x) \frac{u(x)}{B(x)} \,dx \bigg)^{p'} \chi_{[t,\infty)}(s)\,ds\bigg)^{\frac{1}{p'}}}{\|\vp\|_{q',h^{1-q'},(0,\infty)}}.
	\end{align*} 
	
	First consider the case when $p \le q$. By Theorems \ref{thm.Hardy} and \ref{thm.Copson}, we have that
	\begin{align*}
		K \approx & \sup_{h:\, \int_0^x h \le \int_0^x w} \sup_{t \in (0,\infty)} \frac{t^{\frac{1}{p}}}{v_1(t)^{\frac{1}{m}}} \sup_{s \in (0,\infty)} \bigg( \int_s^{\infty} \tau^{-p'} \chi_{[t,\infty)}(\tau)\,d\tau\bigg)^{\frac{1}{p'}} \bigg( \int_0^s h(\tau) u(\tau)^q \,d\tau \bigg)^{\frac{1}{q}}
		\\
		& + \sup_{h:\, \int_0^x h \le \int_0^x \big( \int_{\tau}^{\infty} w
		\big)\,d\tau} \sup_{t \in (0,\infty)} \frac{t^{\frac{1}{p}}}{v_1(t)^{\frac{1}{m}}} \sup_{s \in (0,\infty)} \bigg( \int_0^s \bigg(\frac{B(\tau)}{\tau}\bigg)^{p'} \chi_{[t,\infty)}(\tau)\,d\tau\bigg)^{\frac{1}{p'}} \bigg( \int_s^{\infty} h(\tau) \bigg(\frac{u(\tau)}{B(\tau)}\bigg)^q \,d\tau \bigg)^{\frac{1}{q}}.
	\end{align*} 
	
	Since
		\begin{align*}
		\sup_{s \in (0,\infty)} & \bigg( \int_s^{\infty} \tau^{-p'} \chi_{[t,\infty)}(\tau)\,d\tau\bigg)^{\frac{1}{p'}} \bigg( \int_0^s h(\tau) u(\tau)^q \,d\tau \bigg)^{\frac{1}{q}}
		\\
		 &  = \max \bigg\{\sup_{s \in (0,t)} \bigg( \int_s^{\infty} \tau^{-p'} \chi_{[t,\infty)}(\tau)\,d\tau\bigg)^{\frac{1}{p'}}\bigg( \int_0^s h(\tau) u(\tau)^q \,d\tau \bigg)^{\frac{1}{q}}, 
		 \\
		 & \qquad \qquad \sup_{s \in [t,\infty)} \bigg( \int_s^{\infty} \tau^{-p'} \chi_{[t,\infty)}(\tau)\,d\tau\bigg)^{\frac{1}{p'}}\bigg( \int_0^s h(\tau) u(\tau)^q \,d\tau \bigg)^{\frac{1}{q}}\bigg\}	\\ 
		 & = \max \bigg\{\bigg( \int_t^{\infty} \tau^{-p'} \,d\tau\bigg)^{\frac{1}{p'}}\bigg( \int_0^t h(\tau) u(\tau)^q \,d\tau \bigg)^{\frac{1}{q}}, \sup_{s \in [t,\infty)} \bigg( \int_s^{\infty} \tau^{-p'} \,d\tau\bigg)^{\frac{1}{p'}} \bigg( \int_0^s h(\tau) u(\tau)^q \,d\tau \bigg)^{\frac{1}{q}}\bigg\} \\ 
		 & = \sup_{s \in [t,\infty)} \bigg( \int_s^{\infty} \tau^{-p'} \,d\tau\bigg)^{\frac{1}{p'}} \bigg( \int_0^s h(\tau) u(\tau)^q \,d\tau \bigg)^{\frac{1}{q}} \\ 
		 & \approx \sup_{s \in [t,\infty)} s^{-\frac{1}{p}} \bigg( \int_0^s h(\tau) u(\tau)^q \,d\tau \bigg)^{\frac{1}{q}}
	\end{align*}
	and
	\begin{align*}
		\sup_{s \in (0,\infty)} & \bigg( \int_0^s \bigg(\frac{B(\tau)}{\tau}\bigg)^{p'} \chi_{[t,\infty)}(\tau)\,d\tau\bigg)^{\frac{1}{p'}}\bigg( \int_s^{\infty} h(\tau) \bigg(\frac{u(\tau)}{B(\tau)}\bigg)^q \,d\tau \bigg)^{\frac{1}{q}} \\
		&  = \max \bigg\{\sup_{s \in (0,t)} \bigg( \int_0^s \bigg(\frac{B(\tau)}{\tau}\bigg)^{p'} \chi_{[t,\infty)}(\tau)\,d\tau\bigg)^{\frac{1}{p'}}\bigg( \int_s^{\infty} h(\tau) \bigg(\frac{u(\tau)}{B(\tau)}\bigg)^q \,d\tau \bigg)^{\frac{1}{q}}, \\ 
		& \qquad \qquad \sup_{s \in [t,\infty)} \bigg( \int_0^s \bigg(\frac{B(\tau)}{\tau}\bigg)^{p'} \chi_{[t,\infty)}(\tau)\,d\tau\bigg)^{\frac{1}{p'}} \bigg( \int_s^{\infty} h(\tau) \bigg(\frac{u(\tau)}{B(\tau)}\bigg)^q \,d\tau \bigg)^{\frac{1}{q}}\bigg\}	\\ 
		&  = \max \bigg\{0 , \sup_{s \in [t,\infty)} \bigg( \int_t^s \bigg(\frac{B(\tau)}{\tau}\bigg)^{p'}\,d\tau\bigg)^{\frac{1}{p'}} \bigg( \int_s^{\infty} h(\tau) \bigg(\frac{u(\tau)}{B(\tau)}\bigg)^q \,d\tau \bigg)^{\frac{1}{q}}\bigg\}
		\\ 
		& = \sup_{s \in [t,\infty)} \bigg( \int_t^s \bigg(\frac{B(\tau)}{\tau}\bigg)^{p'}\,d\tau\bigg)^{\frac{1}{p'}} \bigg( \int_s^{\infty} h(\tau) \bigg(\frac{u(\tau)}{B(\tau)}\bigg)^q \,d\tau \bigg)^{\frac{1}{q}}, 
	\end{align*}
	then interchanging the suprema we get that
	\begin{align*}
		K \approx & \sup_{t \in (0,\infty)} \frac{t^{\frac{1}{p}}}{v_1(t)^{\frac{1}{m}}} \sup_{s \in [t,\infty)} s^{-\frac{1}{p}} \sup_{h:\, \int_0^x h \le \int_0^x w} \bigg( \int_0^s h(\tau) u(\tau)^q \,d\tau \bigg)^{\frac{1}{q}}
		\\
		& + \sup_{t \in (0,\infty)} \frac{t^{\frac{1}{p}}}{v_1(t)^{\frac{1}{m}}}\sup_{s \in [t,\infty)} \bigg( \int_t^s \bigg(\frac{B(\tau)}{\tau}\bigg)^{p'}\,d\tau\bigg)^{\frac{1}{p'}} \sup_{h:\, \int_0^x h \le \int_0^x w} \bigg( \int_s^{\infty} h(\tau) \bigg(\frac{u(\tau)}{B(\tau)}\bigg)^q \,d\tau \bigg)^{\frac{1}{q}}.
	\end{align*}

	Finally, by Theorem \ref{transfermon} and Remark \ref{rem}, we arrive at
	\begin{align*}
		K \approx & \sup_{t \in (0,\infty)} \frac{t^{\frac{1}{p}}}{v_1(t)^{\frac{1}{m}}} \sup_{s \in [t,\infty)} s^{-\frac{1}{p}}  \bigg( \int_0^\infty \bigg( \sup_{\tau \in [x,\infty)} u(\tau)^q \chi_{(0,s]}(\tau) \bigg) \,dx \bigg)^{\frac{1}{q}}
		\\
		& + \sup_{t \in (0,\infty)} \frac{t^{\frac{1}{p}}}{v_1(t)^{\frac{1}{m}}}\sup_{s \in [t,\infty)} \bigg( \int_t^s \bigg(\frac{B(\tau)}{\tau}\bigg)^{p'}\,d\tau\bigg)^{\frac{1}{p'}} \bigg( \int_0^{\infty} \bigg( \sup_{\tau \in [x,\infty)} \bigg(\frac{u(\tau)}{B(\tau)}\bigg)^q \chi_{[s, \infty)} (\tau) \bigg) \,dx \bigg)^{\frac{1}{q}} \\
	    \approx & \sup_{t \in (0,\infty)} \frac{t^{\frac{1}{p}}}{v_1(t)^{\frac{1}{m}}}  \sup_{s \in [t,\infty)} s^{-\frac{1}{p}} \bigg(\int_0^{s}\bigg( \sup_{\tau \in [x,s]} u(\tau)^q \bigg) \, w(x)\,dx\bigg) ^{\frac{1}{q}} 
		\\
		& + \sup_{t \in (0,\infty)} \frac{t^{\frac{1}{p}}}{v_1(t)^{\frac{1}{m}}} \sup_{s \in [t,\infty)} \bigg( \int_t^s \bigg(\frac{B(\tau)}{\tau}\bigg)^{p'}\,d\tau\bigg)^{\frac{1}{p'}} \bigg( \sup_{\tau \in [s,\infty)} \frac{u(\tau)}{B(\tau)} \bigg) \bigg( \int_0^{s}  w(x) \, dx \bigg) ^{\frac{1}{q}}
		\\
		& + \sup_{t \in (0,\infty)} \frac{t^{\frac{1}{p}}}{v_1(t)^{\frac{1}{m}}} \sup_{s \in [t,\infty)} \bigg( \int_t^s \bigg(\frac{B(\tau)}{\tau}\bigg)^{p'}\,d\tau\bigg)^{\frac{1}{p'}} \bigg( \int_s^{\infty} \bigg( \sup_{\tau \in [x,\infty)} \bigg( \frac{u(\tau)}{B(\tau)}\bigg)^q \bigg) \, w(x) \, dx \bigg) ^{\frac{1}{q}}.
	\end{align*}

	Next consider the case when $q <  p$. By Theorems \ref{thm.Hardy} and \ref{thm.Copson}, we have that
	\begin{align*}
		K \approx & \sup_{h:\, \int_0^x h \le \int_0^x w} \sup_{t \in (0,\infty)} \frac{t^{\frac{1}{p}}}{v_1(t)^{\frac{1}{m}}} \times \\
		& \qquad \times \bigg(\int_0^{\infty}\bigg( \int_s^{\infty} \tau^{-p'} \chi_{[t,\infty)}(\tau)\,d\tau\bigg)^{\frac{(q-1)p}{p - q}} s^{-p'} \chi_{[t,\infty)}(s) \bigg( \int_0^s h(\tau) u(\tau)^q \,d\tau \bigg)^{\frac{p}{p - q}} \,ds \bigg)^{\frac{p-q}{pq}}
		\\
		& + \sup_{h:\, \int_0^x h \le \int_0^x w} \sup_{t \in (0,\infty)} \frac{t^{\frac{1}{p}}}{v_1(t)^{\frac{1}{m}}} \times \\
		& \qquad \times \bigg(\int_0^{\infty} \bigg( \int_0^s \bigg(\frac{B(\tau)}{\tau}\bigg)^{p'} \chi_{[t,\infty)}(\tau)\,d\tau\bigg)^{\frac{(q-1)p}{p - q}} \bigg(\frac{B(s)}{s}\bigg)^{p'} \chi_{[t,\infty)}(s) \bigg( \int_s^{\infty} h(\tau) \bigg(\frac{u(\tau)}{B(\tau)}\bigg)^q \,d\tau \bigg)^{\frac{p}{p - q}} \, ds\bigg)^{\frac{p-q}{pq}}.
	\end{align*} 

    Since
	\begin{align*}
		& \hspace{-2 cm}\int_0^{\infty}\bigg( \int_s^{\infty} \tau^{-p'} \chi_{[t,\infty)}(\tau)\,d\tau\bigg)^{\frac{(q-1)p}{p - q}} s^{-p'} \chi_{[t,\infty)}(s) \bigg( \int_0^s h(\tau) u(\tau)^q \,d\tau \bigg)^{\frac{p}{p - q}} \,ds
		\\
		= & \int_t^{\infty}\bigg( \int_s^{\infty} \tau^{-p'}\,d\tau\bigg)^{\frac{(q-1)p}{p - q}} s^{-p'} \bigg( \int_0^s h(\tau) u(\tau)^q \,d\tau \bigg)^{\frac{p}{p - q}} \,ds
		\\
		\approx & \int_t^{\infty}s^{\frac{p}{q - p}} \bigg( \int_0^s h(\tau) u(\tau)^q \,d\tau \bigg)^{\frac{p}{p - q}} \,ds
	\end{align*}
	and
	\begin{align*}
		& \hspace{-2 cm} \int_0^{\infty} \bigg( \int_0^s \bigg(\frac{B(\tau)}{\tau}\bigg)^{p'} \chi_{[t,\infty)}(\tau)\,d\tau\bigg)^{\frac{(q-1)p}{p - q}} \bigg(\frac{B(s)}{s}\bigg)^{p'} \chi_{[t,\infty)}(s) \bigg( \int_s^{\infty} h(\tau) \bigg(\frac{u(\tau)}{B(\tau)}\bigg)^q \,d\tau \bigg)^{\frac{p}{p - q}} \, ds
		\\
		= & \int_t^{\infty} \bigg( \int_t^s \bigg(\frac{B(\tau)}{\tau}\bigg)^{p'} \,d\tau\bigg)^{\frac{(q-1)p}{p - q}} \bigg(\frac{B(s)}{s}\bigg)^{p'} \bigg( \int_s^{\infty} h(\tau) \bigg(\frac{u(\tau)}{B(\tau)}\bigg)^q \,d\tau \bigg)^{\frac{p}{p - q}} \, ds
		\\ 
		= & \int_t^{\infty} {\mathcal B}(t,s) \bigg( \int_s^{\infty} h(\tau) \bigg(\frac{u(\tau)}{B(\tau)}\bigg)^q \,d\tau \bigg)^{\frac{p}{p - q}} \, ds,
	\end{align*}
	then
	\begin{align*}
		K \approx & \sup_{h:\, \int_0^x h \le \int_0^x w} \sup_{t \in (0,\infty)} \frac{t^{\frac{1}{p}}}{v_1(t)^{\frac{1}{m}}} \bigg( \int_t^{\infty}s^{\frac{p}{q - p}} \bigg( \int_0^s h(\tau) u(\tau)^q \,d\tau \bigg)^{\frac{p}{p - q}} \,ds\bigg)^{\frac{p-q}{pq}}
		\\
		& +\sup_{h:\, \int_0^x h \le \int_0^x w}
		\sup_{t \in (0,\infty)} \frac{t^{\frac{1}{p}}}{v_1(t)^{\frac{1}{m}}} \bigg(\int_t^{\infty} {\mathcal B}(t,s) \bigg( \int_s^{\infty} h(\tau) \bigg(\frac{u(\tau)}{B(\tau)}\bigg)^q \,d\tau \bigg)^{\frac{p}{p - q}} \, ds\bigg)^{\frac{p-q}{pq}}.
	\end{align*} 

	By duality, applying Fubini's Theorem and interchanging the suprema, we get that 
	\begin{align*}
		K \approx & \sup_{h:\, \int_0^x h \le \int_0^x w} \sup_{t \in (0,\infty)} \frac{t^{\frac{1}{p}}}{v_1(t)^{\frac{1}{m}}} \left( \sup_{\vp \in {\mathfrak M}^+ [t,\infty)} \frac{\int_t^{\infty} \vp(s)  \int_0^s h(\tau) u(\tau)^q \,d\tau \,ds}{\bigg( \int_t^{\infty} \vp(s)^{\frac{p}{q}} s\,ds \bigg)^{\frac{q}{p}}}
		\right)^{\frac{1}{q}}
		\\
		& + \sup_{h:\, \int_0^x h \le \int_0^x w}
		\sup_{t \in (0,\infty)} \frac{t^{\frac{1}{p}}}{v_1(t)^{\frac{1}{m}}}  \left( \sup_{\vp \in {\mathfrak M}^+ [t,\infty)} \frac{\int_t^{\infty} \vp(s) \int_s^{\infty} h(\tau) \bigg(\frac{u(\tau)}{B(\tau)}\bigg)^q \,d\tau \,ds}{\bigg( \int_t^{\infty} \vp(s)^{\frac{p}{q}} {\mathcal B}(t,s)^{\frac{q-p}{q}}\,ds \bigg)^{\frac{q}{p}}} \right)^{\frac{1}{q}} \\
		\approx & \sup_{t \in (0,\infty)} \frac{t^{\frac{1}{p}}}{v_1(t)^{\frac{1}{m}}} \left( \sup_{\vp \in {\mathfrak M}^+ [t,\infty)} \frac{\sup_{h:\, \int_0^x h \le \int_0^x w} \int_0^{t} h(\tau) u(\tau)^q \,d\tau \, \bigg( \int_{t}^{\infty} \vp(s) \,ds \bigg)  }{\bigg( \int_t^{\infty} \vp(s)^{\frac{p}{q}} s\,ds \bigg)^{\frac{q}{p}}} \right)^{\frac{1}{q}}
		\\
		& + \sup_{t \in (0,\infty)} \frac{t^{\frac{1}{p}}}{v_1(t)^{\frac{1}{m}}} \left( \sup_{\vp \in {\mathfrak M}^+ [t,\infty)} \frac{\sup_{h:\, \int_0^x h \le \int_0^x w} \int_t^{\infty} h(\tau) u(\tau)^q \, \bigg( \int_{\tau}^{\infty} \vp(s) \,ds \bigg) \,d\tau }{\bigg( \int_t^{\infty} \vp(s)^{\frac{p}{q}} s\,ds \bigg)^{\frac{q}{p}}} \right)^{\frac{1}{q}}
		\\
		& + \sup_{t \in (0,\infty)} \frac{t^{\frac{1}{p}}}{v_1(t)^{\frac{1}{m}}}  \left( \sup_{\vp \in {\mathfrak M}^+ [t,\infty)} \frac{ \sup_{h:\, \int_0^x h \le \int_0^x w} \int_t^{\infty} h(\tau) \bigg(\frac{u(\tau)}{B(\tau)}\bigg)^q \bigg(\int_t^{\tau} \vp(s)\,ds \bigg) \,d\tau }{\bigg( \int_t^{\infty} \vp(s)^{\frac{p}{q}} {\mathcal B}(t,s)^{\frac{q-p}{q}}\,ds \bigg)^{\frac{q}{p}}} \right)^{\frac{1}{q}}.
	\end{align*}

	On using Theorem \ref{transfermon}, we arrive at
	\begin{align*}
		K \approx & \sup_{t \in (0,\infty)} \frac{t^{\frac{1}{p}}}{v_1(t)^{\frac{1}{m}}} \left( \sup_{\vp \in {\mathfrak M}^+ [t,\infty)} \frac{\bigg( \int_0^{\infty} \bigg( \sup_{\tau \in [x,\infty)} u(\tau)^q \chi_{(0,t]}(\tau) \bigg) \, w(x) \,dx\bigg) \bigg( \int_{t}^{\infty} \vp(s) \,ds \bigg)}{\bigg( \int_t^{\infty} \vp(s)^{\frac{p}{q}} s \,ds \bigg)^{\frac{q}{p}}} \right)^{\frac{1}{q}}
		\\
		& + \sup_{t \in (0,\infty)} \frac{t^{\frac{1}{p}}}{v_1(t)^{\frac{1}{m}}} \left( \sup_{\vp \in {\mathfrak M}^+ [t,\infty)} \frac{ \int_0^{\infty} \bigg( \sup_{\tau \in [x,\infty)} u(\tau)^q \chi_{[t,\infty)}(\tau)\bigg( \int_{\tau}^{\infty} \vp(s) \,ds \bigg) \bigg) \, w(x) \,dx }{\bigg( \int_t^{\infty} \vp(s)^{\frac{p}{q}} s \,ds \bigg)^{\frac{q}{p}}} \right)^{\frac{1}{q}}
		\\
		& + \sup_{t \in (0,\infty)} \frac{t^{\frac{1}{p}}}{v_1(t)^{\frac{1}{m}}}  \left( \sup_{\vp \in {\mathfrak M}^+ [t,\infty)} \frac{ \int_0^{\infty} \bigg( \sup_{\tau \in [x,\infty)} \bigg(\frac{u(\tau)}{B(\tau)}\bigg)^q \chi_{[t,\infty)}(\tau)\bigg(\int_t^{\tau} \vp(s)\,ds \bigg) \bigg)\, w(x) \,dx }{\bigg( \int_t^{\infty} \vp(s)^{\frac{p}{q}} {\mathcal B}(t,s)^{\frac{q-p}{r}}\,ds \bigg)^{\frac{q}{p}}} \right)^{\frac{1}{q}}.
	\end{align*} 

	By Remark \ref{rem}, interchanging the suprema, we get that
	\begin{align*}
		K \approx & \sup_{t \in (0,\infty)} \frac{t^{\frac{1}{p}}}{v_1(t)^{\frac{1}{m}}} \bigg( \int_0^{t} \bigg (\sup_{\tau \in [x,t]} u(\tau)^q \bigg) \, w(x) \,dx\bigg)^{\frac{1}{q}} \left( \sup_{\vp \in {\mathfrak M}^+ [t,\infty)} \frac{\bigg( \int_{t}^{\infty} \vp(s) \,ds \bigg)}{\bigg( \int_t^{\infty} \vp(s)^{\frac{p}{q}} s \,ds \bigg)^{\frac{q}{p}}} \right)^{\frac{1}{q}}
		\\
		& + \sup_{t \in (0,\infty)} \frac{t^{\frac{1}{p}}}{v_1(t)^{\frac{1}{m}}} \bigg(\int_0^{t}\, w(x) \,dx\bigg)^{\frac{1}{q}}  \left( \sup_{\tau \in [t,\infty)} u(\tau)^q \sup_{\vp \in {\mathfrak M}^+ [t,\infty)} \frac{\bigg( \int_{\tau}^{\infty} \vp(s) \,ds \bigg)}{\bigg( \int_t^{\infty} \vp(s)^{\frac{p}{q}} s \,ds \bigg)^{\frac{q}{p}}} \right)^{\frac{1}{q}}
		\\
		& + \sup_{t \in (0,\infty)} \frac{t^{\frac{1}{p}}}{v_1(t)^{\frac{1}{m}}} \left( \sup_{\vp \in {\mathfrak M}^+ [t,\infty)} \frac{ \int_t^{\infty} \bigg( \sup_{\tau \in [x,\infty)} u(\tau)^q \bigg( \int_{\tau}^{\infty} \vp(s) \,ds \bigg) \bigg) \, w(x) \, dx }{\bigg( \int_t^{\infty} \vp(s)^{\frac{p}{q}} s \,ds \bigg)^{\frac{q}{p}}} \right)^{\frac{1}{q}}
		\\
		& + \sup_{t \in (0,\infty)} \frac{t^{\frac{1}{p}}}{v_1(t)^{\frac{1}{m}}} \bigg(\int_0^{t} \, w(x) \,dx\bigg)^{\frac{1}{q}} \left( \sup_{\tau \in [t,\infty)} \bigg(\frac{u(\tau)}{B(\tau)}\bigg)^q \sup_{\vp \in {\mathfrak M}^+ [t,\infty)} \frac{ \bigg(\int_t^{\tau} \vp(s)\,ds \bigg) }{\bigg( \int_t^{\infty} \vp(s)^{\frac{p}{q}} {\mathcal B}(t,s)^{\frac{q-p}{q}}\,ds \bigg)^{\frac{q}{p}}} \right)^{\frac{1}{q}}
		\\
		& + \sup_{t \in (0,\infty)} \frac{t^{\frac{1}{p}}}{v_1(t)^{\frac{1}{m}}}  \left( \sup_{\vp \in {\mathfrak M}^+ [t,\infty)} \frac{ \int_t^{\infty} \bigg( \sup_{\tau \in [x,\infty)} \bigg(\frac{u(\tau)}{B(\tau)}\bigg)^q \bigg(\int_t^{\tau} \vp(s)\,ds \bigg) \bigg) \, w(x) \, dx }{\bigg( \int_t^{\infty} \vp(s)^{\frac{p}{q}} {\mathcal B}(t,s)^{\frac{q-p}{q}}\,ds \bigg)^{\frac{q}{p}}} \right)^{\frac{1}{q}}.
	\end{align*} 

	By duality, we have that
	\begin{align*}
		K \approx & \sup_{t \in (0,\infty)} \frac{t^{\frac{1}{p}}}{v_1(t)^{\frac{1}{m}}} \bigg( \int_0^{t} \bigg( \sup_{\tau \in [x,t]} u(\tau)^q \bigg)  \, w(x) \, dx\bigg)^{\frac{1}{q}} \bigg( \int_{t}^{\infty} s^{\frac{p}{q-p}} \,ds \bigg)^{\frac{p-q}{pq}}
		\\
		& + \sup_{t \in (0,\infty)} \frac{t^{\frac{1}{p}}}{v_1(t)^{\frac{1}{m}}} \bigg(\int_0^{t} \, w(x) \,dx\bigg)^{\frac{1}{q}} \bigg( \sup_{\tau \in [t,\infty)} u(\tau)^q  \bigg( \int_{t}^{\infty} s^{\frac{p}{q-p}} \chi_{[\tau,\infty)}(s)\,ds \bigg)^{\frac{p-q}{p}} \bigg)^{\frac{1}{q}}
		\\
		& + \sup_{t \in (0,\infty)} \frac{t^{\frac{1}{p}}}{v_1(t)^{\frac{1}{m}}} \left( \sup_{\vp \in {\mathfrak M}^+ [t,\infty)} \frac{ \int_t^{\infty} \bigg( \sup_{\tau \in [x,\infty)} u(\tau)^q \bigg( \int_{\tau}^{\infty} \vp(s) \,ds \bigg) \bigg) \, w(x) \, dx }{\bigg( \int_t^{\infty} \vp(s)^{\frac{p}{q}} s \,ds \bigg)^{\frac{q}{p}}} \right)^{\frac{1}{q}}
		\\
		& + \sup_{t \in (0,\infty)} \frac{t^{\frac{1}{p}}}{v_1(t)^{\frac{1}{m}}} \bigg(\int_0^{t} \, w(x) \,dx\bigg)^{\frac{1}{q}} \bigg (\sup_{\tau \in [t,\infty)} \bigg(\frac{u(\tau)}{B(\tau)}\bigg)^q  \bigg( \int_{t}^{\infty} {\mathcal B}(t,s) \chi_{[t,\tau]}(s) \,ds \bigg)^{\frac{p-q}{p}} \bigg)^{\frac{1}{q}}
		\\
		& + \sup_{t \in (0,\infty)} \frac{t^{\frac{1}{p}}}{v_1(t)^{\frac{1}{m}}}  \left( \sup_{\vp \in {\mathfrak M}^+ [t,\infty)} \frac{ \int_t^{\infty} \bigg( \sup_{\tau \in [x,\infty)} \bigg(\frac{u(\tau)}{B(\tau)}\bigg)^q \bigg(\int_t^{\tau} \vp(s)\,ds \bigg) \bigg) \, w(x) \, dx }{\bigg( \int_t^{\infty} \vp(s)^{\frac{p}{q}} {\mathcal B}(t,s)^{\frac{q-p}{q}}\,ds \bigg)^{\frac{q}{p}}} \right)^{\frac{1}{q}}.
	\end{align*} 

	Applying Theorem \ref{thm44b} and Theorem \ref{krepelathm6b} yields
	\begin{align*}
		K \approx & \sup_{t \in (0,\infty)} \frac{1}{v_1(t)^{\frac{1}{m}}} \bigg( \int_0^{t} \bigg( \sup_{\tau \in [x,t]} u(\tau)^q \bigg) \, w(x) \,dx\bigg)^{\frac{1}{q}} 
		\\
		& + \sup_{t \in (0,\infty)} \frac{t^{\frac{1}{p}}}{v_1(t)^{\frac{1}{m}}} \bigg(\int_0^{t} \, w(x) \,dx\bigg)^{\frac{1}{q}} \bigg( \sup_{\tau \in [t,\infty)} u(\tau) \tau^{-\frac{1}{p}} \bigg)
		\\
		& + \sup_{t \in (0,\infty)} \frac{t^{\frac{1}{p}}}{v_1(t)^{\frac{1}{m}}}  \bigg( \int_t^{\infty} \bigg(\sup_{\tau \in [s,\infty)} u(\tau)^{\frac{pq}{p-q}} \tau^{\frac{q}{q-p}}\bigg) \bigg(\int_{t}^{s} w(x) \, dx \bigg)^{\frac{q}{p-q}} w(s) \,ds\bigg)^{\frac{p-q}{pq}}
		\\
		& + \sup_{t \in (0,\infty)} \frac{t^{\frac{1}{p}}}{v_1(t)^{\frac{1}{m}}}  \bigg(\int_t^{\infty} \bigg(\int_{t}^{s} \bigg( \sup_{y\in [x,s]} u(y)^q \bigg) w(x) \,dx\bigg)^{\frac{q}{p-q}} \bigg(\sup_{\tau \in [s,\infty)} u(\tau)^q \tau^{\frac{q}{q-p}}\bigg)w(s)\, ds\bigg)^{\frac{p-q}{pq}}
		\\
		& + \sup_{t \in (0,\infty)} \frac{t^{\frac{1}{p}}}{v_1(t)^{\frac{1}{m}}} \bigg(\int_0^{t} w(x)\,dx\bigg)^{\frac{1}{q}} \bigg( \sup_{\tau \in [t,\infty)} \frac{u(\tau)}{B(\tau)} \bigg( \int_{t}^{\tau} {\mathcal B}(t,s) \,ds \bigg)^{\frac{p-q}{pq}} \bigg)
		\\
		& + \sup_{t \in (0,\infty)} \frac{t^{\frac{1}{p}}}{v_1(t)^{\frac{1}{m}}}  \bigg( \int_t^{\infty} \bigg[\sup_{\tau \in [x,\infty)}\bigg[\sup_{y \in [\tau,\infty)} \bigg(\frac{u(y)}{B(y)}\bigg)^\frac{pq}{p-q}\bigg]\bigg(\int_{t}^{\tau} \mathcal{B}(t,s) \,ds \bigg) \bigg]\bigg(\int_{t}^{x} w(y)\,dy \bigg)^{\frac{q}{p-q}} w(x) \,dx \bigg)^{\frac{p-q}{pq}}
		\\
		& + \sup_{t \in (0,\infty)} \frac{t^{\frac{1}{p}}}{v_1(t)^{\frac{1}{m}}}  \bigg( \int_t^{\infty}\bigg(\int_{x}^{\infty} \bigg[\sup_{\tau \in [y,\infty)}\bigg(\frac{u(\tau)}{B(\tau)}\bigg)^q\bigg]w(y)\, dy\bigg)^{\frac{q}{p-q}} \bigg[\sup_{\tau \in [x,\infty)}\bigg(\frac{u(\tau)}{B(\tau)}\bigg)^q\bigg]\bigg(\int_{t}^{x}\mathcal{B}(t,s) \,ds \bigg) w(x) \,dx\bigg)^{\frac{p-q}{pq}}.
	\end{align*}
	
The proof is completed.
\end{proof}

\begin{theorem}\label{3stresult}
	Let $1 < m < \infty , 0 < p \le 1, 1 < q < \infty$ and $b \in \W (0,\infty) \cap \mp^+ ((0,\infty);\dn)$ be such that the function $B(t)$ satisfies  $0 < B(t) < \infty$ for every $t \in (0,\infty)$. Suppose that $u \in \W(0,\infty) \cap C(0,\infty)$, $v \in \W_{m.p}(0,\infty)$ and $w \in \W(0,\infty)$.
		\begin{itemize}
			\item[i)] If $m \le q$, then
			\begin{align*}
				K \approx & \sup_{t \in (0,\infty)} \bigg( \int_t^{\infty} \frac{v_0(s)}{v_1(s)^{m' + 1}}\,ds \bigg)^{\frac{1}{m'}}  \bigg( \int_0^{t} \bigg( \sup_{\tau \in [x,t]} u(\tau)^q \bigg)\,w(x)\,dx\bigg)^{\frac{1}{q}} 
				\\
				& + \sup_{t \in (0,\infty)} \bigg( \int_0^t \frac{B(s)^{m'}v_0(s)}{v_1(s)^{m' + 1}}\,ds \bigg)^{\frac{1}{m'}} \bigg( \int_0^{t} w(x)\,dx\bigg)^{\frac{1}{q}} \bigg( \sup_{\tau \in[t,\infty)}\frac{u(\tau)}{B(\tau)} \bigg) 
				\\
				& + \sup_{t \in (0,\infty)} \bigg( \int_0^t \frac{B(s)^{m'}v_0(s)}{v_1(s)^{m' + 1}}\,ds \bigg)^{\frac{1}{m'}} \bigg( \int_t^{\infty} \bigg( \sup_{\tau \in[x,\infty)} \bigg( \frac{u(\tau)}{B(\tau)} \bigg)^q \bigg) w(x)\,dx\bigg)^{\frac{1}{q}};
			\end{align*}
			
			\item[ii)] If $q < m$, then
			\begin{align*}
			K \approx & \,  \bigg(\int_{0}^{\infty}\bigg[\sup_{\tau \in [t,\infty)}u(\tau)^\frac{mq}{m-q}\bigg(\int_{\tau}^{\infty}\mathfrak{B}_1 (s) \,ds \bigg)\bigg]\bigg(\int_{0}^{t}w(x)\,dx\bigg)^\frac{q}{m-q}w(t) \, dt\bigg)^\frac{m-q}{mq}
			\\
			& + \bigg(\int_{0}^{\infty}\bigg(\int_{0}^{t}\bigg[\sup_{y \in [x, t]}u(y)^q\bigg]\,w(x)\,dx\bigg)^\frac{q}{m-q} \bigg[\sup_{\tau \in [t,\infty)}u(\tau)^q\bigg(\int_{\tau}^{\infty}\mathfrak{B}_1 (s) \,ds \bigg)\bigg]w(t) \, dt\bigg)^\frac{m-q}{mq}
			\\
			& + \bigg(\int_{0}^{\infty}\bigg[\sup_{\tau \in [t,\infty)}\bigg[\sup_{s \in [\tau,\infty)}\bigg(\frac{u(s)}{B(s)}\bigg)^{\frac{mq}{m-q}} \bigg]\bigg(\int_{0}^{\tau}\mathfrak{B}_2 (s) \,ds \bigg)\bigg] \bigg(\int_{0}^{t}w(x)\,dx\bigg)^\frac{q}{m-q} w(t) \, dt\bigg)^\frac{m-q}{mq}
			\\
			& + \bigg(\int_{0}^{\infty}\bigg(\int_{t}^{\infty}\bigg[\sup_{\tau \in [x,\infty) }\bigg(\frac{u(\tau)}{B(\tau)}\bigg)^q\bigg] w(x)\,dx\bigg)^\frac{q}{m-q} \, \bigg[\sup_{\tau \in [t,\infty)}\bigg(\frac{u(\tau)}{B(\tau)}\bigg)^q\bigg] \bigg(\int_{0}^{t}\mathfrak{B}_2 (s) \,ds \bigg) w(t) \,dt\bigg)^\frac{m-q}{mq},
			\end{align*}  
		where functions $\mathfrak{B}_1$ and $\mathfrak{B}_2$ are defined for all $s \in (0,\infty)$ by
		$$
		\mathfrak{B}_1(s) : = \bigg( \int_s^{\infty} \frac{v_0(t)}{v_1(t)^{m' + 1}}\,dt\bigg)^{\frac{m(q-1)}{m - q}} \frac{v_0(s)}{v_1(s)^{m' + 1}}, \quad 
		\mathfrak{B}_2(s) : = \bigg( \int_0^s \frac{B(t)^{m'}v_0(t)}{v_1(t)^{m' + 1}}\,dt \bigg)^{\frac{m(q-1)}{m - q}} \frac{B(s)^{m'}v_0(s)}{v_1(s)^{m' + 1}},
		$$	
		respectively.
		\end{itemize}
\end{theorem}
\begin{proof}
	By  Lemma \ref{redlemma},Theorem \ref{assosGG}, (iii), and Fubini's Theorem, we have that
		\begin{align*}
		K \approx & \sup_{h:\, \int_0^x h \le \int_0^x w} \sup_{\vp \in \mathfrak{M}^{+}} \frac{1}{\|\vp\|_{q',h^{1-q'},(0,\infty)}}
		\bigg( \int_0^{\infty} \bigg( \frac{1}{t} \int_0^t b(y) \int_y^{\infty} \vp(x) \frac{u(x)}{B(x)} \,dx \,dy \bigg)^{m'} \frac{t^{m'} v_0(t)}{v_1(t)^{m' + 1}}\,dt \bigg)^{\frac{1}{m'}} \\
		\approx &  \sup_{h:\, \int_0^x h \le \int_0^x w}  \sup_{\vp \in \mathfrak{M}^{+}} \frac{1}{\|\vp\|_{q',h^{1-q'},(0,\infty)}} \bigg( \int_0^{\infty} \bigg( \int_0^t \vp(x) u(x)\,dx\bigg)^{m'} \frac{v_0(t)}{v_1(t)^{m' + 1}}\,dt \bigg)^{\frac{1}{m'}} 
		\\
		& + \sup_{h:\, \int_0^x h \le \int_0^x w} \sup_{\vp \in \mathfrak{M}^{+}} \frac{1}{\|\vp\|_{q',h^{1-q'},(0,\infty)}} \bigg( \int_0^{\infty} \bigg(\int_t^{\infty} \vp(x) \frac{u(x)}{B(x)}\,dx\bigg)^{m'} \frac{B(t)^{m'}v_0(t)}{v_1(t)^{m' + 1}}\,dt \bigg)^{\frac{1}{m'}}.
	\end{align*} 
	
	Let $m \le q$. By Theorems \ref{thm.Hardy} and \ref{thm.Copson}, we have that
	\begin{align*}
		K \approx & \sup_{h:\, \int_0^x h \le \int_0^x w} \sup_{t \in (0,\infty)} \bigg( \int_t^{\infty} \frac{v_0(s)}{v_1(s)^{m' + 1}}\,ds \bigg)^{\frac{1}{m'}} \bigg( \int_0^t h(\tau) u(\tau)^q\,d\tau\bigg)^{\frac{1}{q}} 
		\\
		& + \sup_{h:\, \int_0^x h \le \int_0^x w} \sup_{t \in (0,\infty)} \bigg( \int_0^t \frac{B(s)^{m'}v_0(s)}{v_1(s)^{m' + 1}}\,ds \bigg)^{\frac{1}{m'}} \bigg( \int_t^{\infty} h (\tau) \bigg(\frac{u(\tau)}{B(\tau)}\bigg)^q\,d\tau\bigg)^{\frac{1}{q}}.
	\end{align*} 	
	
	Interchanging the suprema, using Theorem \ref{transfermon} and Remark \ref{rem}, we arrive at
	\begin{align*}
		K \approx & \sup_{t \in (0,\infty)} \bigg( \int_t^{\infty} \frac{v_0(s)}{v_1(s)^{m' + 1}}\,ds \bigg)^{\frac{1}{m'}} \sup_{h:\, \int_0^x h \le \int_0^x w}  \bigg( \int_0^{\infty} h(\tau) u(\tau)^q \chi_{(0,t]}(\tau)\,d\tau\bigg)^{\frac{1}{q}} 
		\\
		& + \sup_{t \in (0,\infty)} \bigg( \int_0^t \frac{B(s)^{m'}v_0(s)}{v_1(s)^{m' + 1}}\,ds \bigg)^{\frac{1}{m'}}
		\sup_{h:\, \int_0^x h \le \int_0^x w}  \bigg( \int_0^{\infty} h (\tau) \bigg(\frac{u(\tau)}{B(\tau)}\bigg)^q\chi_{[t,\infty)}(\tau)\,d\tau\bigg)^{\frac{1}{q}} \\
		\approx & \sup_{t \in (0,\infty)} \bigg( \int_t^{\infty} \frac{v_0(s)}{v_1(s)^{m' + 1}}\,ds \bigg)^{\frac{1}{m'}}  \bigg( \int_0^{\infty} \bigg( \sup_{\tau \in [x,\infty)} u(\tau)^q \chi_{(0,t]}(\tau) \bigg)\,w(x)\,dx\bigg)^{\frac{1}{q}} 
		\\
		& + \sup_{t \in (0,\infty)} \bigg( \int_0^t \frac{B(s)^{m'}v_0(s)}{v_1(s)^{m' + 1}}\,ds \bigg)^{\frac{1}{m'}} \bigg( \int_0^{\infty}
		\bigg(  \sup_{\tau \in [x,\infty)} \bigg(\frac{u(\tau)}{B(\tau)}\bigg)^q \chi_{[t,\infty)}(\tau)\bigg)\, w(x)\,dx\bigg)^{\frac{1}{q}} \\
		\approx & \sup_{t \in (0,\infty)} \bigg( \int_t^{\infty} \frac{v_0(s)}{v_1(s)^{m' + 1}}\,ds \bigg)^{\frac{1}{m'}}  \bigg( \int_0^{t} \bigg( \sup_{\tau \in [x,t]} u(\tau)^q \bigg)\,w(x)\,dx\bigg)^{\frac{1}{q}} 
		\\
		& + \sup_{t \in (0,\infty)} \bigg( \int_0^t \frac{B(s)^{m'}v_0(s)}{v_1(s)^{m' + 1}}\,ds \bigg)^{\frac{1}{m'}} \bigg( \int_0^{t} w(x)\,dx\bigg)^{\frac{1}{q}} \bigg( \sup_{\tau \in [t,\infty)} \frac{u(\tau)}{B(\tau)} \bigg)
		\\
		& + \sup_{t \in (0,\infty)} \bigg( \int_0^t \frac{B(s)^{m'}v_0(s)}{v_1(s)^{m' + 1}}\,ds \bigg)^{\frac{1}{m'}} \bigg( \int_t^{\infty} \bigg( \sup_{\tau \in [x,\infty)} \bigg(\frac{u(\tau)}{B(\tau)}\bigg)^q \bigg) w(x)\,dx\bigg)^{\frac{1}{q}}.
	\end{align*} 
	
	Let now $q < m$. By Theorems \ref{thm.Hardy} and \ref{thm.Copson}, we have that
	\begin{align*}
		K \approx & \sup_{h:\, \int_0^x h \le \int_0^x w} \bigg(\int_0^{\infty} \mathfrak{B}_1(s) \bigg( \int_0^s h(\tau) u(\tau)^q \,d\tau \bigg)^{\frac{m}{m - q}} \,ds\bigg)^{\frac{m-q}{mq}} 
		\\
		& + \sup_{h:\, \int_0^x h \le \int_0^x w} \bigg(\int_0^{\infty} \mathfrak{B}_2(s) \bigg( \int_s^{\infty} h(\tau) \bigg(\frac{u(\tau)}{B(\tau)}\bigg)^q \,d\tau \bigg)^{\frac{m}{m - q}} \,ds \bigg)^{\frac{m-q}{mq}}.
	\end{align*} 
	
	By duality, we get that 
	\begin{align*}
		K \approx & \sup_{h:\, \int_0^x h \le \int_0^x w} \left( \sup_{\vp \in \mathfrak{M}^{+}} \frac{\int_0^{\infty} \vp(s) \bigg( \int_0^s h(\tau) u(\tau)^q \,d\tau \bigg)\,ds}{\bigg( \int_0^{\infty} \vp(s)^{\frac{m}{q}} \mathfrak{B}_1(s)^{\frac{q-m}{q}}\,ds \bigg)^{\frac{q}{m}}}
		\right)^{\frac{1}{q}}
		\\
		& + \sup_{h:\, \int_0^x h \le \int_0^x w}  \left( \sup_{\vp \in \mathfrak{M}^{+}} \frac{\int_0^{\infty} \vp(s) \bigg(\int_s^{\infty} h(\tau) \bigg(\frac{u(\tau)}{B(\tau)}\bigg)^q \,d\tau \bigg)\,ds}{\bigg( \int_0^{\infty} \vp(s)^{\frac{m}{q}} \mathfrak{B}_2(s)^{\frac{q-m}{q}}\,ds \bigg)^{\frac{q}{m}}} \right)^{\frac{1}{q}}.
	\end{align*} 
	
	By Fubini's Theorem, interchanging the suprema, on using Theorem \ref{transfermon}, we arrive at
	\begin{align*}
		K \approx & \left( \sup_{\vp \in \mathfrak{M}^{+}} \frac{ \sup_{h:\, \int_0^x h \le \int_0^x w} \int_0^{\infty} h(\tau) u(\tau)^q \bigg( \int_{\tau}^{\infty} \vp(s) \,ds \bigg) \,d\tau}{\bigg( \int_0^{\infty} \vp(s)^{\frac{m}{q}} \mathfrak{B}_1(s)^{\frac{q-m}{q}}\,ds \bigg)^{\frac{q}{m}}} \right)^{\frac{1}{q}}
		\\
		& + \left( \sup_{\vp \in \mathfrak{M}^{+}} \frac{\sup_{h:\, \int_0^x h \le \int_0^x w} \int_0^{\infty} h(\tau) \bigg(\frac{u(\tau)}{B(\tau)}\bigg)^q \bigg(\int_0^{\tau} \vp(s) \,ds\bigg)\,d\tau }{\bigg( \int_0^{\infty} \vp(s)^{\frac{m}{q}} \mathfrak{B}_2(s)^{\frac{q-m}{q}}\,ds \bigg)^{\frac{q}{m}}} \right)^{\frac{1}{q}} \\
        \approx & \left( \sup_{\vp \in \mathfrak{M}^{+}} \frac{ \int_0^{\infty} \bigg(\sup_{\tau \in [x,\infty)} u(\tau)^q \bigg( \int_{\tau}^{\infty} \vp(s) \,ds \bigg) \bigg) w(x)\,dx }{\bigg(\int_0^{\infty} \vp(s)^{\frac{m}{q}} \mathfrak{B}_1(s)^{\frac{q-m}{q}}\,ds \bigg)^{\frac{q}{m}}}\right)^{\frac{1}{q}}
		\\
		& + \left( \sup_{\vp \in \mathfrak{M}^{+}} \frac{ \int_0^{\infty} \bigg( \sup_{\tau \in [x,\infty)} \bigg(\frac{u(\tau)}{B(\tau)}\bigg)^q \bigg(\int_0^{\tau} \vp(s) \,ds\bigg) \bigg) w(x)\,dx}{\bigg( \int_0^{\infty} \vp(s)^{\frac{m}{q}} \mathfrak{B}_2(s)^{\frac{q-m}{q}}\,ds \bigg)^{\frac{q}{m}}}\right)^{\frac{1}{q}}.
	\end{align*} 
	
	Applying Theorem \ref{thm44b} and Theorem \ref{krepelathm6b} yields
	\begin{align*}
		K \approx & \,  \bigg(\int_{0}^{\infty}\bigg[\sup_{\tau \in [t,\infty)}u(\tau)^\frac{mq}{m-q}\bigg(\int_{\tau}^{\infty}\mathfrak{B}_1 (s) \,ds \bigg)\bigg]\bigg(\int_{0}^{t}w(x)\,dx\bigg)^\frac{q}{m-q}w(t) \, dt\bigg)^\frac{m-q}{mq}
		\\
		& + \bigg(\int_{0}^{\infty}\bigg(\int_{0}^{t}\bigg[\sup_{y \in [x, t]}u(y)^q\bigg]\,w(x)\,dx\bigg)^\frac{q}{m-q} \bigg[\sup_{\tau \in [t,\infty)}u(\tau)^q\bigg(\int_{\tau}^{\infty}\mathfrak{B}_1 (s) \,ds \bigg)\bigg]w(t) \, dt\bigg)^\frac{m-q}{mq}
		\\
		& + \bigg(\int_{0}^{\infty}\bigg[\sup_{\tau \in [t,\infty)}\bigg[\sup_{s \in [\tau,\infty)}\bigg(\frac{u(s)}{B(s)}\bigg)^{\frac{mq}{m-q}} \bigg]\bigg(\int_{0}^{\tau}\mathfrak{B}_2 (s) \,ds \bigg)\bigg] \bigg(\int_{0}^{t}w(x)\,dx\bigg)^\frac{q}{m-q} w(t) \, dt\bigg)^\frac{m-q}{mq}
		\\
		& + \bigg(\int_{0}^{\infty}\bigg(\int_{t}^{\infty}\bigg[\sup_{\tau \in [x,\infty) }\bigg(\frac{u(\tau)}{B(\tau)}\bigg)^q\bigg] w(x)\,dx\bigg)^\frac{q}{m-q} \, \bigg[\sup_{\tau \in [t,\infty)}\bigg(\frac{u(\tau)}{B(\tau)}\bigg)^q\bigg] \bigg(\int_{0}^{t}\mathfrak{B}_2 (s) \,ds \bigg) w(t) \,dt\bigg)^\frac{m-q}{mq}.
	\end{align*}
	
	The proof is completed. 
\end{proof}

\begin{theorem}\label{4stresult}
	Let $1 < m < \infty , 1 < p < \infty, 1 < q < \infty$ and $b \in \W (0,\infty) \cap \mp^+ ((0,\infty);\dn)$ be such that the function $B(t)$ satisfies  $0 < B(t) < \infty$ for every $t \in (0,\infty)$. Assume that $u \in \W(0,\infty) \cap C(0,\infty)$, $v \in \W_{m,p}(0,\infty)$ and $w \in \W(0,\infty)$.
	Suppose that
    \begin{gather*}
    \int_0^t v_2(s)\,ds < \infty, \quad \int_t^{\infty} s^{-\frac{m'}{p}} v_2(s)\,ds < \infty, \qquad 0 < \int_0^t \bigg( \int_s^t \bigg( \frac{B(y)}{y} \bigg)^{p'}\,dy \bigg)^{\frac{m'}{p'}} v_2(s)\,ds < \infty, \qquad t \in (0,\infty), \\
    \int_0^1 s^{-\frac{m'}{p}} v_2(s)\,ds = \int_1^{\infty} v_2(s) \,ds = \infty,
    \end{gather*}
    where the function $v_2$ is defined by
	\begin{equation*}
		v_2(t) : = \frac{t^{\frac{m'}{p'}}v_0(t)}{v_1(t)^{m' + 1}}, \qquad t \in (0,\infty).
	\end{equation*} 
		
	\item[i)] If  $\max\{p,\,m\} \le q$, then
	\begin{align*}
		K \approx & \sup_{t \in (0,\infty)}\bigg( \int_0^t \bigg( \int_s^t \bigg( \frac{B(y)}{y} \bigg)^{p'} \, dy \bigg)^{\frac{m'}{p'}} v_2(s) \, ds \bigg)^{\frac{1}{m'}} 
		\bigg( \int_0^{t} w(x) \, dx 	\bigg)^{\frac{1}{q}} \bigg( \sup_{\tau \in [t,\infty)} \frac{u(\tau)}{B(\tau)} \bigg)
		\\
		& + \sup_{t \in (0,\infty)}\bigg( \int_0^t \bigg( \int_s^t \bigg( \frac{B(y)}{y} \bigg)^{p'} \, dy \bigg)^{\frac{m'}{p'}} v_2(s) \, ds \bigg)^{\frac{1}{m'}} 
		\bigg(\int_t^{\infty} \bigg( \sup_{\tau \in [x,\infty)} \bigg( \frac{u(\tau)}{B(\tau)} \bigg)^q \bigg) w(x) \,dx \bigg)^{\frac{1}{q}}
		\\
		& + \sup_{t \in (0,\infty)}  \bigg( \int_0^{\infty} \bigg( \frac{1}{s + t}\bigg)^{\frac{m'}{p}} v_2(s) \, ds \bigg)^{\frac{1}{m'}} 
		\bigg(\int_0^{t} \bigg( \sup_{\tau \in [x,t]} u(\tau)^q \bigg) w(x) \,dx \bigg)^{\frac{1}{q}}.
	\end{align*}

	\item[ii)] If $m \le q < p$, then
	\begin{align*}
		K \approx & \sup_{t \in (0,\infty)}\bigg( \int_0^t \bigg( \int_s^t \bigg( \frac{B(y)}{y} \bigg)^{p'} \, dy \bigg)^{\frac{m'}{p'}} v_2(s) \, ds \bigg)^{\frac{1}{m'}} 
		\bigg( \int_0^{t} w(x) \, dx 	\bigg)^{\frac{1}{q}} \bigg( \sup_{\tau \in [t,\infty)} \frac{u(\tau)}{B(\tau)} \bigg) 
		\\
		& + \sup_{t \in (0,\infty)}\bigg( \int_0^t \bigg( \int_s^t \bigg( \frac{B(y)}{y} \bigg)^{p'} \, dy \bigg)^{\frac{m'}{p'}} v_2(s) \, ds \bigg)^{\frac{1}{m'}} 
		\bigg(\int_t^{\infty} \bigg( \sup_{\tau \in [x,\infty)} \bigg( \frac{u(\tau)}{B(\tau)} \bigg)^q \bigg) w(x) \,dx \bigg)^{\frac{1}{q}}
		\\
		& + \sup_{t \in (0,\infty)}\bigg( \int_0^t v_2 \bigg)^{\frac{1}{m'}} \bigg(\int_0^{t} w(x) \, dx\bigg)^{\frac{1}{q}} \bigg( \sup_{\tau \in [t,\infty)} \frac{u(\tau)}{B(\tau)} \bigg( \int_{t}^{\tau} {\mathcal B}(t,s) \,ds \bigg)^{\frac{p-q}{pq}} \bigg)
		\\
		& + \sup_{t \in (0,\infty)}\bigg( \int_0^t v_2 \bigg)^{\frac{1}{m'}}  \bigg( \int_t^{\infty} \bigg[\sup_{\tau \in [s,\infty)}\bigg[\sup_{x \in [\tau,\infty)} \bigg(\frac{u(x)}{B(x)}\bigg)^\frac{pq}{p-q}\bigg]\bigg(\int_{t}^{\tau} \mathcal{B}(t,y) \,dy\bigg)\bigg]\bigg(\int_{t}^{s}w(y)\,dy \bigg)^{\frac{q}{p-q}} w(s) \,ds\bigg)^{\frac{p-q}{pq}}
		\\
		& + \sup_{t \in (0,\infty)}\bigg( \int_0^t v_2 \bigg)^{\frac{1}{m'}}  \bigg( \int_t^{\infty}\bigg(\int_{s}^{\infty} \bigg[\sup_{\tau \in [x,\infty)}\bigg(\frac{u(\tau)}{B(\tau)}\bigg)^q\bigg]w(x) \, dx\bigg)^{\frac{q}{p-q}} \bigg[\sup_{\tau \in [s,\infty)}\bigg(\frac{u(\tau)}{B(\tau)}\bigg)^q\bigg]\bigg(\int_{t}^{s}\mathcal{B}(t,y) \,dy \bigg) w(s) \,ds\bigg)^{\frac{p-q}{pq}}
		\\
		& + \sup_{t \in (0,\infty)} \bigg(\int_{0}^{\infty} \bigg(\frac{t}{s + t}\bigg)^{\frac{m'}{p}} v_2(s) \, ds\bigg)^\frac{1}{m'}\bigg(\int_{0}^{\infty}\bigg[\sup_{\tau \in [t,\infty)}u(\tau)^\frac{pq}{p-q}(\tau+t)^\frac{q}{q-p}\bigg]\bigg(\int_{0}^{t}w(x)\,dx\bigg)^\frac{q}{p-q} w(t)\, dt\bigg)^\frac{p-q}{pq}
		\\
		& + \sup_{t \in (0,\infty)} \bigg(\int_{0}^{\infty} \bigg(\frac{t}{s + t}\bigg)^{\frac{m'}{p}} v_2(s) \, ds\bigg)^\frac{1}{m'}\bigg(\int_{0}^{\infty}\bigg(\int_{0}^{t}\bigg[\sup_{y \in [x, t]}u(y)^q\bigg]\,w(x)\,dx\bigg)^\frac{q}{p-q} \bigg[\sup_{\tau \in [t,\infty)}u(\tau)^q(\tau+t)^\frac{q}{q-p}\bigg] w(t) \, dt\bigg)^\frac{p-q}{pq}
	\end{align*}
\end{theorem}
\begin{proof}
	By  Lemma \ref{redlemma}, Theorem \ref{assosGG}, (iv), and Fubini's Theorem, we have that
	\begin{align*}
		K \approx & \sup_{h:\, \int_0^x h \le \int_0^x w} \bigg( \int_0^{\infty} \bigg( \int_t^{\infty} \bigg( \frac{1}{s} \int_0^s  b(y) \int_y^{\infty} \vp(x) \frac{u(x)}{B(x)} \,dx \,dy \bigg)^{p'}\,ds \bigg)^{\frac{m'}{p'}} v_2 (t) \,dt \bigg)^{\frac{1}{m'}}  \\
		\approx & \sup_{h:\, \int_0^x h \le \int_0^x w}  \sup_{\vp \in \mathfrak{M}^{+}} \frac{1}{\|\vp\|_{q',h^{1-q'},(0,\infty)}} \bigg( \int_0^{\infty} \bigg( \int_t^{\infty} \bigg( 	\frac{B(s)}{s}  \int_s^{\infty} \vp(x) \frac{u(x)}{B(x)} \,dx  \bigg)^{p'}\,ds \bigg)^{\frac{m'}{p'}} v_2(t)\,dt \bigg)^{\frac{1}{m'}}
		\\
		& + \sup_{h:\, \int_0^x h \le \int_0^x w} \sup_{\vp \in \mathfrak{M}^{+}} \frac{1}{\|\vp\|_{q',h^{1-q'},(0,\infty)}} \bigg( \int_0^{\infty} \bigg( \int_t^{\infty} \bigg( \frac{1}{s}  	\int_0^s \vp(x) u(x) \,dx  \bigg)^{p'}\,ds \bigg)^{\frac{m'}{p'}} v_2(t) \,dt \bigg)^{\frac{1}{m'}}
		=: \, A + B.  
	\end{align*} 
	
	Let $p \le q$ and $m \le q$. We first estimate $A$. By Theorem \ref{krepick}, we get that
	\begin{equation*}
		A \approx \sup_{h:\, \int_0^x h \le \int_0^x w} \sup_{t \in (0,\infty)}\bigg( \int_0^t \bigg( \int_s^t \bigg( \frac{B(y)}{y} \bigg)^{p'}\,dy \bigg)^{\frac{m'}{p'}}
		v_2(s)\,ds \bigg)^{\frac{1}{m'}}\bigg( \int_t^{\infty} h(x) \bigg( \frac{u(x)}{B(x)} \bigg)^q \,dx \bigg)^{\frac{1}{q}} .
	\end{equation*}
	
	By interchanging the suprema, applying Theorem \ref{transfermon} and Remark \ref{rem}, we obtain that
	\begin{align*}
		A \approx & \sup_{t \in (0,\infty)}\bigg( \int_0^t \bigg( \int_s^t \bigg( \frac{B(y)}{y} \bigg)^{p'}\,dy \bigg)^{\frac{m'}{p'}}
		v_2(s)\,ds \bigg)^{\frac{1}{m'}} \sup_{h:\, \int_0^x h \le \int_0^x w} \bigg( \int_0^{\infty} h(x) \bigg( \frac{u(x)}{B(x)} \bigg)^q \chi_{[t,\infty)}(x) \,dx \bigg)^{\frac{1}{q}}
		\\
		\approx & \sup_{t \in (0,\infty)}\bigg( \int_0^t \bigg( \int_s^t \bigg( \frac{B(y)}{y} \bigg)^{p'}\,dy \bigg)^{\frac{m'}{p'}}
		v_2(s)\,ds \bigg)^{\frac{1}{m'}} \bigg(\int_0^{\infty} \bigg( \sup_{\tau \in [x,\infty)} \bigg( \frac{u(\tau)}{B(\tau)} \bigg)^q \chi_{[t,\infty)}(\tau) \bigg) w(x) \,dx \bigg)^{\frac{1}{q}} \\
		\approx  & \sup_{t \in (0,\infty)}\bigg( \int_0^t \bigg( \int_s^t \bigg( \frac{B(y)}{y} \bigg)^{p'} \, dy \bigg)^{\frac{m'}{p'}} v_2(s) \, ds \bigg)^{\frac{1}{m'}} 
		\bigg( \sup_{\tau \in [t,\infty)} \frac{u(\tau)}{B(\tau)} \bigg) \bigg( \int_0^{t} w(x) \, dx 	\bigg)^{\frac{1}{q}}
		\\
		& + \sup_{t \in (0,\infty)}\bigg( \int_0^t \bigg( \int_s^t \bigg( \frac{B(y)}{y} \bigg)^{p'} \, dy \bigg)^{\frac{m'}{p'}} v_2(s) \, ds \bigg)^{\frac{1}{m'}} 
		\bigg(\int_t^{\infty} \bigg( \sup_{\tau \in [x,\infty)} \bigg( \frac{u(\tau)}{B(\tau)} \bigg)^q \bigg) w(x) \,dx \bigg)^{\frac{1}{q}}.
	\end{align*}
	
	Now we estimate $B$. Note that, by Theorem \ref{gks}, the following equivalency holds:
	\begin{align*}
        B \approx & \sup_{h:\, \int_0^x h \le \int_0^x w}  \sup_{t \in (0,\infty)}  \bigg( \int_t^{\infty} s^{-\frac{m'}{p}} v_2(s) \, ds \bigg)^{\frac{1}{m'}}  \bigg( \int_0^t h\, u^q \bigg)^{\frac{1}{q}}
        \\
        & + \sup_{h:\, \int_0^x h \le \int_0^x w} \sup_{t \in (0,\infty)} t^{-\frac{1}{p}}\bigg(\int_0^t v_2 (s)\,ds \bigg)^\frac{1}{m'} \bigg( \int_0^t h\, u^q \bigg)^{\frac{1}{q}} \\
        \approx &  \sup_{h:\, \int_0^x h \le \int_0^x w} \sup_{t \in (0,\infty)}  \bigg( \int_0^{\infty} \bigg( \frac{1}{s +t} \bigg)^{\frac{m'}{p}}v_2(s) \, ds \bigg)^{\frac{1}{m'}}  \bigg( \int_0^t h u^q \bigg)^{\frac{1}{q}}.
	\end{align*}
	
	By interchanging suprema, applying Theorem \ref{transfermon} and Remark \ref{rem}, we have that
	\begin{align*}
		B \approx &  \sup_{t \in (0,\infty)}  \bigg( \int_0^{\infty} \bigg( \frac{1}{s + t}\bigg)^{\frac{m'}{p}}v_2(s) \,ds \bigg)^{\frac{1}{m'}} \left( \sup_{h:\, \int_0^x h \le \int_0^x w} \int_0^{\infty} h(x) u(x)^q \chi_{(0,t]}(x) \,dx \right)^{\frac{1}{q}}.
		\\
		\approx & \sup_{t \in (0,\infty)}  \bigg( \int_0^{\infty} \bigg( \frac{1}{s + t}\bigg)^{\frac{m'}{p}} v_2(s) \, ds \bigg)^{\frac{1}{m'}} \bigg(\int_0^{\infty} \bigg( \sup_{\tau \in [x,\infty)} u(\tau)^q \chi_{(0,t]}(\tau) \bigg) w(x) \,dx \bigg)^{\frac{1}{q}} \\
	    \approx & \sup_{t \in (0,\infty)}  \bigg( \int_0^{\infty} \bigg( \frac{1}{s + t}\bigg)^{\frac{m'}{p}} v_2(s) \, ds \bigg)^{\frac{1}{m'}} 
		\bigg(\int_0^{t} \bigg( \sup_{\tau \in [x,t]} u(\tau)^q \bigg) w(x) \,dx \bigg)^{\frac{1}{q}}.
	\end{align*}
	
	Combining the estimates for $A$ and $B$, we obtain that
	\begin{align*}
		K \approx & \sup_{t \in (0,\infty)}\bigg( \int_0^t \bigg( \int_s^t \bigg( \frac{B(y)}{y} \bigg)^{p'} \, dy \bigg)^{\frac{m'}{p'}} v_2(s) \, ds \bigg)^{\frac{1}{m'}} 
		\bigg( \sup_{t \le \tau} \frac{u(\tau)}{B(\tau)} \bigg) \bigg( \int_0^{t} w(x) \, dx 	\bigg)^{\frac{1}{q}}
		\\
		& + \sup_{t \in (0,\infty)}\bigg( \int_0^t \bigg( \int_s^t \bigg( \frac{B(y)}{y} \bigg)^{p'} \, dy \bigg)^{\frac{m'}{p'}} v_2(s) \, ds \bigg)^{\frac{1}{m'}} 
		\bigg(\int_t^{\infty} \bigg( \sup_{\tau \in [x,\infty)} \bigg( \frac{u(\tau)}{B(\tau)} \bigg)^q \bigg) w(x) \,dx \bigg)^{\frac{1}{q}}
		\\
		& + \sup_{t \in (0,\infty)}  \bigg( \int_0^{\infty} \bigg( \frac{1}{s + t}\bigg)^{\frac{m'}{p}} v_2(s) \, ds \bigg)^{\frac{1}{m'}} 
		\bigg(\int_0^{t} \bigg( \sup_{\tau \in [x,t]} u(\tau)^q \bigg) w(x) \,dx \bigg)^{\frac{1}{q}}.
	\end{align*}
	
	Let $m \le q < p$. By Theorem \ref{krepick}, we get that
	\begin{align*}
		A \approx &  \sup_{h:\, \int_0^x h \le \int_0^x w}  \sup_{t \in (0,\infty)}\bigg( \int_0^t v_2(s)	\bigg( \int_s^t \bigg( \frac{B(y)}{y} \bigg)^{p'}\,dy \bigg)^{\frac{m'}{p'}}
		\,ds \bigg)^{\frac{1}{m'}}\bigg( \int_t^{\infty} h(x) \bigg( \frac{u(x)}{B(x)} \bigg)^q \,dx \bigg)^{\frac{1}{q}}
		\\
		& + \sup_{h:\, \int_0^x h \le \int_0^x w} \sup_{t \in (0,\infty)}\bigg( \int_0^t 
		v_2 (s) \, ds \bigg)^{\frac{1}{m'}}\bigg( \int_t^{\infty} {\mathcal B}(t,s)\bigg(\int_{s}^{\infty}h(x) \bigg( \frac{u(x)}{B(x)} \bigg)^q \,dx\bigg)^\frac{p}{p-q} \, ds \bigg)^{\frac{p-q}{pq}}
		\\
		&=: A_1 + A_2.
	\end{align*}
	
	Since in the previous case $ A_1$ was evaluated as
	\begin{align*}
		A_1 \approx  & \sup_{t \in (0,\infty)}\bigg( \int_0^t \bigg( \int_s^t \bigg( \frac{B(y)}{y} \bigg)^{p'} \, dy \bigg)^{\frac{m'}{p'}} v_2(s) \, ds \bigg)^{\frac{1}{m'}} 
		\bigg( \int_0^{t} w(x) \, dx 	\bigg)^{\frac{1}{q}} \bigg( \sup_{\tau \in [t,\infty)} \frac{u(\tau)}{B(\tau)} \bigg) 
		\\
		& + \sup_{t \in (0,\infty)}\bigg( \int_0^t \bigg( \int_s^t \bigg( \frac{B(y)}{y} \bigg)^{p'} \, dy \bigg)^{\frac{m'}{p'}} v_2 (s) \, ds \bigg)^{\frac{1}{m'}} 
		\bigg(\int_t^{\infty} \bigg( \sup_{\tau \in [x,\infty)} \bigg( \frac{u(\tau)}{B(\tau)} \bigg)^q \bigg) w(x) \,dx \bigg)^{\frac{1}{q}},
	\end{align*}
	we continue with the estimate of $A_2$. 
	
	By duality and Fubini's theorem, we have that
	\begin{align*}
		A_2 = & \sup_{h:\, \int_0^x h \le \int_0^x w} \sup_{t \in (0,\infty)}\bigg( \int_0^t v_2 (s) \, ds \bigg)^{\frac{1}{m'}}\left\{\sup_{\psi \in {\mathfrak M}^+ [t,\infty)}\frac{\int_t^{\infty} \psi(s) \int_{s}^{\infty}h(x) \bigg( \frac{u(x)}{B(x)} \bigg)^q \,dx \, ds}{\bigg(\int_t^{\infty} \psi(s)^{\frac{p}{q}}{\mathcal B}(t,s)^{\frac{q-p}{q}}\, ds\bigg)^{\frac{q}{p}}} \right\}^{\frac{1}{q}} \\
		= & \sup_{h:\, \int_0^x h \le \int_0^x w}\sup_{t \in (0,\infty)}\bigg( \int_0^t 
		v_2 (s) \, ds \bigg)^{\frac{1}{m'}}\left\{\sup_{\psi \in {\mathfrak M}^+ [t,\infty)}\frac{\int_t^{\infty} h(x) \bigg( \frac{u(x)}{B(x)} \bigg)^q \int_{t}^{x}\psi(s) \, ds \, dx} {\bigg(\int_t^{\infty} \psi(s)^{\frac{p}{q}}{\mathcal B}(t,s)^{\frac{q-p}{q}}\, ds \bigg)^{\frac{q}{p}}} \right\}^{\frac{1}{q}}.
	\end{align*}
	
	By interchanging suprema, duality and Theorem \ref{transfermon}, in view of Remark \ref{rem}, we arrive at
	\begin{align*}
		A_2 = &  \sup_{t \in (0,\infty)}\bigg( \int_0^t v_2 (s) \, ds \bigg)^{\frac{1}{m'}}\left\{\sup_{\psi \in {\mathfrak M}^+ [t,\infty)}\frac{\sup_{h:\, \int_0^x h \le \int_0^x w}\int_0^{\infty} h(x) \bigg( \frac{u(x)}{B(x)}  \bigg)^q \chi_{[t,\infty)}(x) \int_{t}^{x}\psi(s) \, ds\,dx}{\bigg(\int_t^{\infty} \psi(s)^{\frac{p}{q}}{\mathcal B}(t,s)^{\frac{q-p}{q}}\, ds\bigg)^{\frac{q}{p}}}
		\right\}^{\frac{1}{q}}
		\\
		= & \sup_{t \in (0,\infty)}\bigg( \int_0^t v_2 (s) \, ds \bigg)^{\frac{1}{m'}} \left\{\sup_{\psi \in {\mathfrak M}^+ [t,\infty)}\frac{\int_0^{\infty} \bigg( \sup_{\tau \in [x,\infty)} \bigg( \frac{u(\tau)}{B(\tau)}  \bigg)^q \chi_{[t,\infty)}(\tau) \bigg(\int_{t}^{\tau}\psi(s)\, ds\bigg) \bigg) w(x)\,dx}{\bigg(\int_t^{\infty} \psi(s)^{\frac{p}{q}}{\mathcal B}(t,s)^{\frac{q-p}{q}}\, ds \bigg)^{\frac{q}{p}}}
		\right\}^{\frac{1}{q}}
		\\
		\approx & \sup_{t \in (0,\infty)}\bigg( \int_0^t v_2 (s) \, ds \bigg)^{\frac{1}{m'}} \bigg(\int_0^{t} w(x) \, dx\bigg)^{\frac{1}{q}} \left\{ \sup_{\tau \in [t,\infty)} \bigg(\frac{u(\tau)}{B(\tau)}\bigg)^q \sup_{\psi \in {\mathfrak M}^+ [t,\infty)}\frac{ \int_{t}^{\tau} \psi (s) \,ds}{\bigg(\int_t^{\infty} \psi(s)^{\frac{p}{q}}{\mathcal B}(t,s)^{\frac{q-p}{q}}\, ds \bigg)^{\frac{q}{p}}}
		\right\}^{\frac{1}{q}} 
		\\
		& + \sup_{t \in (0,\infty)}\bigg( \int_0^t v_2 (s) \, ds \bigg)^{\frac{1}{m'}} \left( \sup_{\psi \in {\mathfrak M}^+ [t,\infty)} \frac{ \int_t^{\infty} \bigg( \sup_{\tau \in [x,\infty)} \bigg(\frac{u(\tau)}{B(\tau)}\bigg)^q \bigg(\int_t^{\tau} \psi(s)\,ds \bigg) \bigg) w(x) \,dx }{\bigg( \int_t^{\infty} \psi (s)^{\frac{p}{q}} {\mathcal B}(t,s)^{\frac{q-p}{q}}\,ds \bigg)^{\frac{q}{p}}} \right)^{\frac{1}{q}} \\
		\approx & \sup_{t \in (0,\infty)}\bigg( \int_0^t v_2 (s) \, ds \bigg)^{\frac{1}{m'}} \bigg(\int_0^{t} w(x) \, dx\bigg)^{\frac{1}{q}} \sup_{\tau \in [t,\infty)} \bigg(\frac{u(\tau)}{B(\tau)}\bigg)  \bigg( \int_{t}^{\infty} {\mathcal B}(t,s) \chi_{[t,\tau]}(s) \,ds \bigg)^{\frac{p-q}{pq}}
		\\
		& + \sup_{t \in (0,\infty)}\bigg( \int_0^t v_2 (s) \, ds \bigg)^{\frac{1}{m'}} \left( \sup_{\vp \in {\mathfrak M}^+ [t,\infty)} \frac{ \int_t^{\infty} \bigg( \sup_{\tau \in [x,\infty)} \bigg(\frac{u(\tau)}{B(\tau)}\bigg)^q \bigg(\int_t^{\tau} \vp(s)\,ds \bigg) \bigg) w(x) \,dx }{\bigg( \int_t^{\infty} \vp(s)^{\frac{p}{q}} {\mathcal B}(t,s)^{\frac{q-p}{q}}\,ds \bigg)^{\frac{q}{p}}} \right)^{\frac{1}{q}}.
	\end{align*}
	
	Applying Theorem \ref{thm44b} yields
	\begin{align*}
		A_2 \approx & \sup_{t \in (0,\infty)}\bigg( \int_0^t v_2 (s) \, ds \bigg)^{\frac{1}{m'}} \bigg(\int_0^{t} w(x) \, dx\bigg)^{\frac{1}{q}} \bigg(\sup_{\tau \in [t,\infty)} \frac{u(\tau)}{B(\tau)}  \bigg( \int_{t}^{\tau} {\mathcal B}(t,s) \,ds \bigg)^{\frac{p-q}{pq}} \bigg)
		\\
		& + \sup_{t \in (0,\infty)}\bigg( \int_0^t v_2 (s) \, ds \bigg)^{\frac{1}{m'}}  \bigg( \int_t^{\infty} \bigg[\sup_{\tau \in [s,\infty)}\bigg[\sup_{x \in [\tau,\infty)} \bigg(\frac{u(x)}{B(x)}\bigg)^\frac{pq}{p-q}\bigg]\bigg(\int_{t}^{\tau} \mathcal{B}(t,y) \,dy\bigg)\bigg]\bigg(\int_{t}^{s}w(y)\,dy \bigg)^{\frac{q}{p-q}} w(s) \,ds\bigg)^{\frac{p-q}{pq}}
		\\
		& + \sup_{t \in (0,\infty)}\bigg( \int_0^t v_2 (s) \, ds \bigg)^{\frac{1}{m'}}  \bigg( \int_t^{\infty}\bigg(\int_{s}^{\infty} \bigg[\sup_{\tau \in [x,\infty)}\bigg(\frac{u(\tau)}{B(\tau)}\bigg)^q\bigg]w(x) \, dx\bigg)^{\frac{q}{p-q}} \bigg[\sup_{\tau \in [s,\infty)}\bigg(\frac{u(\tau)}{B(\tau)}\bigg)^q\bigg]\bigg(\int_{t}^{s}\mathcal{B}(t,y) \,dy \bigg) w(s) \,ds\bigg)^{\frac{p-q}{pq}}.
	\end{align*}
	
	Combining the estimates for $A_1$ and $A_2$, we arrive at
	\begin{align*}
		A \approx & \sup_{t \in (0,\infty)}\bigg( \int_0^t \bigg( \int_s^t \bigg( \frac{B(y)}{y} \bigg)^{p'} \, dy \bigg)^{\frac{m'}{p'}} v_2(s) \, ds \bigg)^{\frac{1}{m'}} 
		\bigg( \sup_{\tau \in [t,\infty)} \frac{u(\tau)}{B(\tau)} \bigg) \bigg( \int_0^{t} w(x) \, dx 	\bigg)^{\frac{1}{q}}
		\\
		& + \sup_{t \in (0,\infty)}\bigg( \int_0^t \bigg( \int_s^t \bigg( \frac{B(y)}{y} \bigg)^{p'} \, dy \bigg)^{\frac{m'}{p'}} v_2(s) \, ds \bigg)^{\frac{1}{m'}} 
		\bigg(\int_t^{\infty} \bigg( \sup_{\tau \in [x,\infty)} \bigg( \frac{u(\tau)}{B(\tau)} \bigg)^q \bigg) w(x) \,dx \bigg)^{\frac{1}{q}}
		\\
		& + \sup_{t \in (0,\infty)}\bigg( \int_0^t v_2 (s) \, ds \bigg)^{\frac{1}{m'}} \bigg(\int_0^{t} w(x) \, dx\bigg)^{\frac{1}{q}} \bigg( \sup_{\tau \in [t,\infty)} \frac{u(\tau)}{B(\tau)} \bigg( \int_{t}^{\tau} {\mathcal B}(t,s) \,ds \bigg)^{\frac{p-q}{pq}} \bigg)
		\\
		& + \sup_{t \in (0,\infty)}\bigg( \int_0^t v_2 (s) \, ds \bigg)^{\frac{1}{m'}}  \bigg( \int_t^{\infty} \bigg[\sup_{\tau \in [s,\infty)}\bigg[\sup_{x \in [\tau,\infty)} \bigg(\frac{u(x)}{B(x)}\bigg)^\frac{pq}{p-q}\bigg]\bigg(\int_{t}^{\tau} \mathcal{B}(t,y) \,dy\bigg)\bigg]\bigg(\int_{t}^{s}w(y)\,dy \bigg)^{\frac{q}{p-q}} w(s) \,ds\bigg)^{\frac{p-q}{pq}}
		\\
		& + \sup_{t \in (0,\infty)}\bigg( \int_0^t v_2 (s) \, ds \bigg)^{\frac{1}{m'}}  \bigg( \int_t^{\infty}\bigg(\int_{s}^{\infty} \bigg[\sup_{\tau \in [x,\infty)}\bigg(\frac{u(\tau)}{B(\tau)}\bigg)^q\bigg]w(x) \, dx\bigg)^{\frac{q}{p-q}} \bigg[\sup_{\tau \in [s,\infty)}\bigg(\frac{u(\tau)}{B(\tau)}\bigg)^q\bigg]\bigg(\int_{t}^{s}\mathcal{B}(t,y) \,dy \bigg) w(s) \,ds\bigg)^{\frac{p-q}{pq}}.
	\end{align*}
	
	Now we estimate $B$. By Theorem \ref {gks} we have
	\begin{align*}
		B \approx & \sup_{h:\, \int_0^x h \le \int_0^x w}  \sup_{t \in (0,\infty)}  \bigg( \int_t^{\infty} s^{-\frac{m'}{p}} v_2(s) \, ds \bigg)^{\frac{1}{m'}}  \bigg( \int_0^t h\, u^q \bigg)^{\frac{1}{q}}
		\\
		& + \sup_{h:\, \int_0^x h \le \int_0^x w} \sup_{t \in (0,\infty)} \bigg(\int_0^t v_2 (s) \, ds \bigg)^\frac{1}{m'} \bigg(\int_t^\infty s^\frac{q}{q-p} \bigg( \int_0^s h\, u^q\bigg)^\frac{q}{p-q} h(s) u(s)^q \, ds\bigg)^\frac{p-q}{pq}.
	\end{align*}
	
	Since
	$$ 
	\bigg( \int_{0}^{t} hu^q \bigg)^\frac{p}{p-q} = \int_{0}^{t} d\,\bigg( \int_{0}^{s} hu^q \bigg)^\frac{p}{p-q} \approx \int_{0}^{t} \bigg( \int_{0}^{s} hu^q \bigg)^\frac{q}{p-q} h(s)u(s)^q \, ds,
	$$
	then we get that
	\begin{align*}
		B \approx &  \sup_{h:\, \int_0^x h \le \int_0^x w} \sup_{t \in (0,\infty)}  \bigg( \int_t^{\infty} s^{-\frac{m'}{p}} v_2(s) \, ds \bigg)^{\frac{1}{m'}}  \bigg( \int_{0}^{t} \bigg( \int_{0}^{s} hu^q \bigg)^\frac{q}{p-q} h(s)u(s)^q \, ds \bigg)^{\frac{p-q}{pq}}
		\\
		& + \sup_{h:\, \int_0^x h \le \int_0^x w} \sup_{t \in (0,\infty)} \bigg(\int_0^t v_2 (s) \, ds \bigg)^\frac{1}{m'} \bigg(\int_t^\infty s^\frac{q}{q-p} \bigg( \int_0^s h\, u^q\bigg)^\frac{q}{p-q} h(s) u(s)^q \, ds\bigg)^\frac{p-q}{pq}.
	\end{align*}
	
	By Lemma \ref{gluing.lem.0}, we arrive at
	\begin{align*}
		B &  \approx \sup_{h:\, \int_0^x h \le \int_0^x w} \sup_{t \in (0,\infty)} \bigg(\int_{0}^{\infty} \bigg(\frac{t}{s + t}\bigg)^{\frac{m'}{p}} v_2(s) \, ds\bigg)^\frac{1}{m'} \bigg(\int_{0}^{\infty}\bigg(\frac{1}{s + t}\bigg)^\frac{q}{p-q} \bigg( \int_{0}^{s} hu^q \bigg)^\frac{q}{p-q} h(s)u(s)^q \, ds\bigg)^\frac{p-q}{pq}.
	\end{align*}
	
	Applying Theorem \ref {thm.IBP.0} to the last integral, we obtain that 
	\begin{align*}
		\int_{0}^{\infty}\bigg(\frac{1}{s + t}\bigg)^\frac{q}{p-q} \bigg( \int_{0}^{s} hu^q \bigg)^\frac{q}{p-q} h(s)u(s)^q \, ds & \\
		& \hspace{-3cm} \approx \int_{0}^{\infty} \bigg( \int_{0}^{s} hu^q \bigg)^\frac{p}{p-q}  d\,\bigg(-\bigg(\frac{1}{s + t}\bigg)^\frac{q}{p-q}\bigg) \\
		& \hspace{-3cm} \approx -\int_{0}^{\infty} \bigg( \int_{0}^{s} hu^q \bigg)^\frac{p}{p-q} \,\bigg(\frac{1}{s + t}\bigg)^\frac{2q-p}{p-q}\bigg(-\frac{1}{(s + t)^2}\bigg)\, ds \\
		& \hspace{-3cm} = \int_{0}^{\infty} \bigg( \int_{0}^{s} hu^q \bigg)^\frac{p}{p-q} \,\bigg(\frac{1}{s + t}\bigg)^\frac{p}{p-q}\, ds.
	\end{align*}
	
	By duality, we obtain that
	\begin{align*}
		B \approx &  \sup_{h:\, \int_0^x h \le \int_0^x w} \sup_{t \in (0,\infty)} \bigg(\int_{0}^{\infty} \bigg(\frac{t}{s + t}\bigg)^{\frac{m'}{p}} v_2(s) \, ds\bigg)^\frac{1}{m'} \bigg(\int_{0}^{\infty} \bigg( \int_{0}^{s} hu^q \bigg)^\frac{p}{p-q} \,\bigg(\frac{1}{s + t}\bigg)^\frac{p}{p-q}\, ds\bigg)^\frac{p-q}{pq}
		\\
		\approx &  \sup_{h:\, \int_0^x h \le \int_0^x w}   \sup_{t \in (0,\infty)} \bigg(\int_{0}^{\infty} \bigg(\frac{t}{s + t}\bigg)^{\frac{m'}{p}} v_2(s) \, ds\bigg)^\frac{1}{m'} \left( \sup_{\psi \in \mathfrak{M}^{+}} \frac{\int_0^{\infty}   \psi (s) \bigg( \int_0^s h u^q \bigg) \,  ds }{\bigg( \int_0^{\infty} \psi(s)^{\frac{p}{q}}  \, (s+t)^{\frac{p}{q}} \,  ds \bigg)^{\frac{q}{p}}} \right)^{\frac{1}{q}}. 
	\end{align*}
	
	Interchanging suprema, Fubini theorem and Theorem \ref{transfermon} implies
	\begin{align*}
		B \approx  &\sup_{t \in (0,\infty)} \bigg(\int_{0}^{\infty} \bigg(\frac{t}{s + t}\bigg)^{\frac{m'}{p}} v_2(s) \, ds\bigg)^\frac{1}{m'} \left( \sup_{\psi \in \mathfrak{M}^{+}} \frac{ \sup_{h:\, \int_0^x h \le \int_0^x w} \int_0^{\infty} h(\tau) u(\tau)^q \bigg( \int_{\tau}^{\infty} \psi(s) \,ds \bigg) \,d\tau}{\bigg( \int_0^{\infty} \psi(s)^{\frac{p}{q}} (s+t)^{\frac{p}{q}}\,ds \bigg)^{\frac{q}{p}}} \right)^{\frac{1}{q}}
		\\
		\approx  &\sup_{t \in (0,\infty)} \bigg(\int_{0}^{\infty} \bigg(\frac{t}{s + t}\bigg)^{\frac{m'}{p}} v_2(s) \, ds\bigg)^\frac{1}{m'} \left( \sup_{\psi \in \mathfrak{M}^{+}} \frac{ \int_0^{\infty} \bigg( \sup_{\tau \in [x,\infty)} u(\tau)^q \bigg( \int_{\tau}^{\infty} \psi(s) \,ds \bigg) \bigg) w(x)\,dx } {\bigg(\int_0^{\infty} \psi(s)^{\frac{p}{q}}(s+t)^{\frac{p}{q}} \,ds \bigg)^{\frac{q}{p}}}\right)^{\frac{1}{q}}.
	\end{align*}
	
	Applying Theorem \ref{krepelathm6b} yields
	\begin{align*}
		B \approx & \sup_{t \in (0,\infty)} \bigg(\int_{0}^{\infty} \bigg(\frac{t}{s + t}\bigg)^{\frac{m'}{p}} v_2(s) \, ds\bigg)^\frac{1}{m'}\bigg(\int_{0}^{\infty}\bigg[\sup_{\tau \in [t,\infty)}u(\tau)^\frac{pq}{p-q}(\tau+t)^\frac{q}{q-p}\bigg]\bigg(\int_{0}^{t}w(x)\,dx\bigg)^\frac{q}{p-q} w(t)\, dt\bigg)^\frac{p-q}{pq}
		\\
		& + \sup_{t \in (0,\infty)} \bigg(\int_{0}^{\infty} \bigg(\frac{t}{s + t}\bigg)^{\frac{m'}{p}} v_2(s) \, ds\bigg)^\frac{1}{m'}\bigg(\int_{0}^{\infty}\bigg(\int_{0}^{t}\bigg[\sup_{y \in [x, t]}u(y)^q\bigg]\,w(x)\,dx\bigg)^\frac{q}{p-q} \bigg[\sup_{\tau \in [t,\infty)}u(\tau)^q(\tau+t)^\frac{q}{q-p}\bigg] w(t) \, dt\bigg)^\frac{p-q}{pq}.
	\end{align*}
	
	Finally, if we combine the estimates of $A$ and $B$, then we get the result as follows:
	\begin{align*}
		K \approx & \sup_{t \in (0,\infty)}\bigg( \int_0^t \bigg( \int_s^t \bigg( \frac{B(y)}{y} \bigg)^{p'} \, dy \bigg)^{\frac{m'}{p'}} v_2(s) \, ds \bigg)^{\frac{1}{m'}} 
		\bigg( \int_0^{t} w(x) \, dx 	\bigg)^{\frac{1}{q}} \bigg( \sup_{\tau \in [t,\infty)} \frac{u(\tau)}{B(\tau)} \bigg) 
		\\
		& + \sup_{t \in (0,\infty)}\bigg( \int_0^t \bigg( \int_s^t \bigg( \frac{B(y)}{y} \bigg)^{p'} \, dy \bigg)^{\frac{m'}{p'}} v_2(s) \, ds \bigg)^{\frac{1}{m'}} 
		\bigg(\int_t^{\infty} \bigg( \sup_{\tau \in [x,\infty)} \bigg( \frac{u(\tau)}{B(\tau)} \bigg)^q \bigg) w(x) \,dx \bigg)^{\frac{1}{q}}
		\\
		& + \sup_{t \in (0,\infty)}\bigg( \int_0^t v_2 (s) \, ds \bigg)^{\frac{1}{m'}} \bigg(\int_0^{t} w(x) \, dx\bigg)^{\frac{1}{q}} \bigg( \sup_{\tau \in [t,\infty)} \frac{u(\tau)}{B(\tau)}\bigg( \int_{t}^{\tau} {\mathcal B}(t,s) \,ds \bigg)^{\frac{p-q}{pq}}
		\bigg) \\
		& + \sup_{t \in (0,\infty)}\bigg( \int_0^t v_2 (s) \, ds \bigg)^{\frac{1}{m'}}  \bigg( \int_t^{\infty} \bigg[\sup_{\tau \in [s,\infty)}\bigg[\sup_{x \in [\tau,\infty)} \bigg(\frac{u(x)}{B(x)}\bigg)^\frac{pq}{p-q}\bigg]\bigg(\int_{t}^{\tau} \mathcal{B}(t,y) \,dy\bigg)\bigg]\bigg(\int_{t}^{s}w(y)\,dy \bigg)^{\frac{q}{p-q}} w(s) \,ds\bigg)^{\frac{p-q}{pq}}
		\\
		& + \sup_{t \in (0,\infty)}\bigg( \int_0^t v_2 (s) \, ds \bigg)^{\frac{1}{m'}}  \bigg( \int_t^{\infty}\bigg(\int_{s}^{\infty} \bigg[\sup_{\tau \in [x,\infty)}\bigg(\frac{u(\tau)}{B(\tau)}\bigg)^q\bigg]w(x) \, dx\bigg)^{\frac{q}{p-q}} \bigg[\sup_{\tau \in [s,\infty)}\bigg(\frac{u(\tau)}{B(\tau)}\bigg)^q\bigg]\bigg(\int_{t}^{s}\mathcal{B}(t,y) \,dy \bigg) w(s) \,ds\bigg)^{\frac{p-q}{pq}}
		\\
		& + \sup_{t \in (0,\infty)} \bigg(\int_{0}^{\infty} \bigg(\frac{t}{s + t}\bigg)^{\frac{m'}{p}} v_2(s) \, ds\bigg)^\frac{1}{m'}\bigg(\int_{0}^{\infty}\bigg[\sup_{\tau \in [t,\infty)}u(\tau)^\frac{pq}{p-q}(\tau+t)^\frac{q}{q-p}\bigg]\bigg(\int_{0}^{t}w(x)\,dx\bigg)^\frac{q}{p-q} w(t)\, dt\bigg)^\frac{p-q}{pq}
		\\
		& + \sup_{t \in (0,\infty)} \bigg(\int_{0}^{\infty} \bigg(\frac{t}{s + t}\bigg)^{\frac{m'}{p}} v_2(s) \, ds\bigg)^\frac{1}{m'}\bigg(\int_{0}^{\infty}\bigg(\int_{0}^{t}\bigg[\sup_{y \in [x, t]}u(y)^q\bigg]\,w(x)\,dx\bigg)^\frac{q}{p-q} \bigg[\sup_{\tau \in [t,\infty)}u(\tau)^q(\tau+t)^\frac{q}{q-p}\bigg] w(t) \, dt\bigg)^\frac{p-q}{pq}.
	\end{align*}
\end{proof}


\section{Boundedness of $M_{\phi,\Lambda^{\alpha}(b)}$ from $\GG(p,m,v)$ into $\Lambda^q (w)$}\label{BofMF}

In this section we formulate and prove the reduction theorem for the boundedness of $M_{\phi,\Lambda^{\alpha}(b)}$ from $\GG(p,m,v)$ into $\Lambda^q (w)$ and calculate the best constant in the inequality 
\begin{equation}\label{opnorm of M}
\bigg( \int_0^{\infty} \big[ \big(M_{\phi,\Lambda^{\alpha}(b)}f\big)^* (x)\big]^q  w(x)\,dx\bigg)^{\frac{1}{q}} \le C \bigg( \int_0^{\infty} \bigg( \int_0^x [f^* (\tau)]^{p}\,d\tau \bigg)^{\frac{m}{p}} v(x)\,dx \bigg)^{\frac{1}{m}},
\end{equation}
which is required to hold for all  $f \in \mp (\rn)$.

A function $\phi: (0,\infty) \rightarrow (0,\infty)$ is said to satisfy the $\Delta_2$ - condition, denoted $\phi \in \Delta_2$, if for some $C > 0$
$$
\phi (2t) \le C \, \phi (t) \quad \mbox{for all} \quad 0 < t < \infty.
$$

A function $\phi: (0,\infty) \rightarrow (0,\infty)$ is said to be
quasi-increasing, if for some $C > 0$
$$
\phi (t_1) \le C \phi (t_2),
$$
whenever $0 < t_1 \le t_2 < \infty$.

A function $\phi: (0,\infty) \rightarrow (0,\infty)$ is said to
satisfy the $Q_r$-condition, $0 < r < \infty$, denoted $\phi \in
Q_r(0,\infty)$, if for some constant $C > 0$
$$
\phi \bigg(\sum_{i=1}^n t_i \bigg) \le C \bigg( \sum_{i=1}^n
\phi(t_i)^r \bigg)^{1/r},
$$
for every finite set of non-negative real numbers
$\{t_1,\ldots,t_n\}$.

\begin{theorem}\label{main.reduc.thm}
	Let $0 < p,\, m,\, q < \infty$, $0 < \alpha \le r < \infty$ and $v,\,w \in \W (0,\infty)$. Assume that $\phi \in Q_{r}(0,\infty)$ is a quasi-increasing function. Suppose that $b \in \W (0,\infty)$ is such that $0 < B(t) < \infty$ for all $t > 0$, $B \in \Delta_2$, $B(\infty) = \infty$ and $B(t) / t^{\alpha / r}$ is quasi-increasing. Then inequality \eqref{opnorm of M} holds for all $f \in \mp (\rn)$ if and only if the inequality
	$$
	\bigg( \int_0^{\infty} \big[ T_{B/\phi^{\alpha},b} h^* (x) \big]^{\frac{q}{\alpha}} w(x)\,dx\bigg)^{\frac{1}{q}} \le C \bigg( \int_0^{\infty} \bigg( \int_0^x [h^* (\tau)]^{\frac{p}{\alpha}}\,d\tau \bigg)^{\frac{m}{p}} v(x)\,dx \bigg)^{\frac{1}{m}}
	$$
	holds for all $h \in \mp (\rn)$.
\end{theorem}

\begin{proof}
Assume that the inequality 	
$$
\bigg( \int_0^{\infty} \big[ \big(M_{\phi,\Lambda^{\alpha}(b)}f\big)^* (x)\big]^{q} w(x)\,dx\bigg)^{\frac{1}{q}} \le C \bigg( \int_0^{\infty} \bigg( \int_0^x [f^* (\tau)]^{p} \,d\tau \bigg)^{\frac{m}{p}} v(x)\,dx \bigg)^{\frac{1}{m}}
$$
holds for all $f \in \mp (\rn)$. 

Denote by $\mf^{\rad,\dn}(\rn)$ the set of all measurable, non-negative,
radially decreasing functions on $\rn$, that is,
$$
\mf^{\rad,\dn}(\rn) : = \{f \in \mf(\rn):\, f(x) = h(|x|),\,x \in \rn
~\mbox{with}~ h \in \mp^+ ((0,\infty);\dn)\}.
$$

Recall that the inequality
$$
\big(M_{\phi,\Lambda^{\alpha}(b)}g\big)^* (x) \ge C \, \sup_{\tau \in [x,\infty)} \phi (\tau)^{-1} \bigg( \int_0^{\tau} [g^* (y)]^{\alpha} b(y)\,dy\bigg)^{\frac{1}{\alpha}}
$$	
holds for all $g \in \mf^{\rad,\dn}(\rn)$ with constant $C > 0$ independent of $g$ and $x$ (cf. \cite[Lemma 3.12]{musbil}). 

Thus the inequality
$$
\bigg( \int_0^{\infty} \bigg[ \sup_{\tau \in [x,\infty)} \phi (\tau)^{-1} \bigg( \int_0^{\tau} [g^*(y)]^{\alpha} b(y)\,dy\bigg)^{\frac{1}{\alpha}}\bigg]^q w(x)\,dx\bigg)^{\frac{1}{q}} \le C \bigg( \int_0^{\infty} \bigg( \int_0^x [g^* (\tau)]^{p}\,d\tau \bigg)^{\frac{m}{p}} v(x)\,dx \bigg)^{\frac{1}{m}}
$$
holds for all $g \in \mf^{\rad,\dn}(\rn)$, which evidently can be rewritten as follows
$$
\bigg( \int_0^{\infty} \big[ (T_{B/\phi^{\alpha},b} g^*) (x) \big]^{\frac{q}{\alpha}} w(x)\,dx\bigg)^{\frac{1}{q}} \le C \bigg( \int_0^{\infty} \bigg( \int_0^x [g^* (\tau)]^{\frac{p}{\alpha}}\,d\tau \bigg)^{\frac{m}{p}} v(x)\,dx \bigg)^{\frac{1}{m}}, \quad g \in \mf^{\rad,\dn}(\rn).
$$

Since for any $h \in \mp (\rn)$ there exists $g \in \mf^{\rad,\dn}(\rn)$ such that $g^* = h^*$, then the ineqaulity
$$
\bigg( \int_0^{\infty} \big[ (T_{B/\phi^{\alpha},b} h^*) (x) \big]^{\frac{q}{\alpha}} w(x)\,dx\bigg)^{\frac{1}{q}} \le C \bigg( \int_0^{\infty} \bigg( \int_0^x [h^* (\tau)]^{\frac{p}{\alpha}}\,d\tau \bigg)^{\frac{m}{p}} v(x)\,dx \bigg)^{\frac{1}{m}}
$$
holds for all $h \in \mp (\rn)$, as well.

Now assume that the inequality
$$
\bigg( \int_0^{\infty} \big[ (T_{B/\phi^{\alpha},b} h^*) (x) \big]^{\frac{q}{\alpha}} w(x)\,dx\bigg)^{\frac{1}{q}} \le C \bigg( \int_0^{\infty} \bigg( \int_0^x [h^* (\tau)]^{\frac{p}{\alpha}}\,d\tau \bigg)^{\frac{m}{p}} v(x)\,dx \bigg)^{\frac{1}{m}}
$$
holds for all $h \in \mp (\rn)$.

Obviously, the last inequality is equivalent to the inequality
$$
\bigg( \int_0^{\infty} \bigg[ \sup_{\tau \in [x,\infty)} \phi (\tau)^{-1} \bigg( \int_0^{\tau} [f^* (y)]^{\alpha} b(y)\,dy\bigg)^{\frac{1}{\alpha}}\bigg]^{q} w(x)\,dx \bigg)^{\frac{1}{q}} \le C \bigg( \int_0^{\infty} \bigg( \int_0^x [f^* (\tau)]^{p}\,d\tau \bigg)^{\frac{m}{p}} v(x)\,dx \bigg)^{\frac{1}{m}}
$$
for all $f \in \mp (\rn)$. 

Recall that the inequality
$$
(M_{\phi,\Lambda^{\alpha}(b)}f)^* (t) \le C \sup_{\tau \in [t,\infty)} \phi (\tau)^{-1} \bigg( \int_0^{\tau} [f^* (y)]^{\alpha}b(y)\,dy\bigg)^{\frac{1}{\alpha}}
$$
holds for all $f \in \mp (\rn)$ (cf. \cite[Corollary 3.6]{musbil}).

Consequently, the inequality
$$
\bigg( \int_0^{\infty} \big[ \big(M_{\phi,\Lambda^{\alpha}(b)}f\big)^* (x)\big]^{q} w(x)\,dx\bigg)^{\frac{1}{q}} \le C \bigg( \int_0^{\infty} \bigg( \int_0^x [f^* (\tau)]^{p}\,d\tau \bigg)^{\frac{m}{p}} v(x)\,dx \bigg)^{\frac{1}{m}}
$$
holds for all $f \in \mp (\rn)$, as well.

The proof is completed.


\end{proof}

Combining Theorem \ref{main.reduc.thm} with Theorems \ref{1stresult}, \ref{2stresult}, \ref{3stresult} and \ref{4stresult}, respectively, we get the following four statements.
\begin{theorem} \label{main1stresult}
	Let	$0 < m \le \alpha \le r < \infty$, $0 < p \le \alpha < q < \infty$ and $b \in \W (0,\infty) \cap \mp^+ ((0,\infty);\dn)$ be such that the function $B(t)$ satisfies  $0 < B(t) < \infty$ for every $t \in (0,\infty)$, $B \in \Delta_2$, $B(\infty) = \infty$ and $B(t) / t^{\alpha / r}$ is quasi-increasing. Moreover, let $\phi \in \W (0,\infty) \cap C(0,\infty)$ be such that $\phi \in Q_{r}(0,\infty)$ is a quasi-increasing function. Assume that $v \in \W_{m,p}(0,\infty)$ and $w \in \W(0,\infty)$.
	Then
	\begin{align*}
    \|M_{\phi,\Lambda^{\alpha}(b)}\|_{\GG(p,m,v) \rightarrow \Lambda^q (w)} & \\
    & \hspace{-3cm} \approx \, \sup_{t \in (0,\infty)} \frac{1}{v_1(t)^{\frac{1}{m}}} \bigg(\int_0^{t} \bigg( \sup_{\tau \in [x,t]} \frac{B(\tau)^{\frac{1}{\alpha}}}{\phi(\tau)} \bigg)^q \, w(x)\,dx\bigg) ^{\frac{1}{q}}
	\\
	& \hspace{-2.5cm} + \, \sup_{t \in (0,\infty)} \frac{B(t)^{\frac{1}{\alpha}}}{v_1(t)^{\frac{1}{m}}}  \bigg( \sup_{\tau \in [t,\infty)} \frac{1}{\phi(\tau)} \bigg) \bigg( \int_0^{t}  w(x) \, dx \bigg) ^{\frac{1}{q}}
	\\
	& \hspace{-2.5cm} + \sup_{t \in (0,\infty)} \frac{B(t)^{\frac{1}{\alpha}}}{v_1(t)^{\frac{1}{m}}}\bigg( \int_t^{\infty} \bigg( \sup_{\tau \in [x,\infty)} \frac{1}{\phi(\tau)}\bigg)^q \, w(x) \, dx \bigg) ^{\frac{1}{q}},
	\end{align*}
	where $v_1$ is defined by \eqref{defof_u}
\end{theorem}	

\begin{theorem} \label{main2stresult}
	Let	$0 < m \le \alpha \le r < \infty$, $\alpha < \min\{p,\,q\} < \infty$ and $b \in \W (0,\infty) \cap \mp^+ ((0,\infty);\dn)$ be such that the function $B(t)$ satisfies  $0 < B(t) < \infty$ for every $t \in (0,\infty)$, $B \in \Delta_2$, $B(\infty) = \infty$ and $B(t) / t^{\alpha / r}$ is quasi-increasing. Moreover, let $\phi \in \W (0,\infty) \cap C(0,\infty)$ be such that $\phi \in Q_{r}(0,\infty)$ is a quasi-increasing function. Assume that $v \in \W_{m,p}(0,\infty)$ and $w \in \W(0,\infty)$.
	\begin{itemize}
	\item[i)] If $p \le q$, then
	\begin{align*}
	\|M_{\phi,\Lambda^{\alpha}(b)}\|_{\GG(p,m,v) \rightarrow \Lambda^q (w)} & \\
	& \hspace{-3cm} \approx \sup_{t \in (0,\infty)} \frac{t^{\frac{1}{p}}}{v_1(t)^{\frac{1}{m}}}  \sup_{s \in [t,\infty)} s^{-\frac{1}{p}} \bigg(\int_0^{s} \bigg( \sup_{\tau \in [x,s]} \frac{B(\tau)^{\frac{1}{\alpha}}}{\phi(\tau)} \bigg)^q \, w(x)\,dx\bigg) ^{\frac{1}{q}}
	\\
	& \hspace{-2.5cm} + \sup_{t \in (0,\infty)} \frac{t^{\frac{1}{p}}}{v_1(t)^{\frac{1}{m}}} \sup_{s \in [(t,\infty)} \bigg( \int_t^s \bigg(\frac{B(\tau)}{\tau}\bigg)^{\frac{p}{p - \alpha}}\,d\tau\bigg)^{\frac{p - \alpha}{p \alpha}} \bigg( \sup_{\tau \in [s,\infty)} \frac{1}{\phi(\tau)} \bigg) \bigg( \int_0^{s}  w(x) \, dx \bigg) ^{\frac{1}{q}}
	\\
	& \hspace{-2.5cm} + \sup_{t \in (0,\infty)} \frac{t^{\frac{1}{p}}}{v_1(t)^{\frac{1}{m}}} \sup_{s \in [t,\infty)} \bigg( \int_t^s \bigg(\frac{B(\tau)}{\tau}\bigg)^{\frac{p}{p - \alpha}}\,d\tau\bigg)^{\frac{p - \alpha}{p \alpha}} \bigg( \int_s^{\infty} \bigg( \sup_{\tau \in [x,\infty)} \frac{1}{\phi(\tau)} \bigg)^q \, w(x) \, dx \bigg) ^{\frac{1}{q}};
	\end{align*}
	
	\item[ii)] If $q <  p$, then
	\begin{align*}
	\|M_{\phi,\Lambda^{\alpha}(b)}\|_{\GG(p,m,v) \rightarrow \Lambda^q (w)} & \\
	& \hspace{-3cm} \approx \, \sup_{t \in (0,\infty)} \frac{1}{v_1(t)^{\frac{1}{m}}} \bigg( \int_0^{t} \bigg( \sup_{\tau \in [x,t]} \frac{B(\tau)^{\frac{1}{\alpha}}}{\phi(\tau)} \bigg)^q \, w(x) \,dx\bigg)^{\frac{1}{q}} 
	\\
	& \hspace{-2.5cm} + \sup_{t \in (0,\infty)} \frac{t^{\frac{1}{p}}}{v_1(t)^{\frac{1}{m}}} \bigg(\int_0^{t} \, w(x) \,dx\bigg)^{\frac{1}{q}} \bigg( \sup_{\tau \in [t,\infty)} \frac{B(\tau)^{\frac{1}{\alpha}}}{\phi(\tau)} \tau^{-\frac{1}{p}} \bigg)
	\\
	& \hspace{-2.5cm} + \sup_{t \in (0,\infty)} \frac{t^{\frac{1}{p}}}{v_1(t)^{\frac{1}{m}}}  \bigg( \int_t^{\infty} \bigg(\sup_{\tau \in [s,\infty)} \bigg( \frac{B(\tau)^{\frac{1}{\alpha}}}{\phi(\tau)}\bigg)^{\frac{pq}{p-q}} \tau^{\frac{q}{q-p}}\bigg) \bigg(\int_{t}^{s} w(x) \, dx \bigg)^{\frac{q}{p-q}} w(s) \,ds\bigg)^{\frac{p-q}{pq}} \\
	& \hspace{-2.5cm} + \sup_{t \in (0,\infty)} \frac{t^{\frac{1}{p}}}{v_1(t)^{\frac{1}{m}}}  \bigg(\int_t^{\infty} \bigg(\int_{t}^{s} \bigg( \sup_{y\in [x,s]} \frac{B(y)^{\frac{1}{\alpha}}}{\phi(y)} \bigg)^q w(x) \,dx\bigg)^{\frac{q}{p-q}} \bigg(\sup_{\tau \in [s,\infty)} \bigg( \frac{B(\tau)^{\frac{1}{\alpha}}}{\phi(\tau)}\bigg)^q \tau^{\frac{q}{q-p}}\bigg) w(s)\, ds\bigg)^{\frac{p-q}{pq}}
	\\
	& \hspace{-2.5cm} + \sup_{t \in (0,\infty)} \frac{t^{\frac{1}{p}}}{v_1(t)^{\frac{1}{m}}} \bigg(\int_0^{t} w(x)\,dx\bigg)^{\frac{1}{q}} \bigg( \sup_{\tau \in [t,\infty)} \frac{1}{\phi(\tau)} \bigg( \int_{t}^{\tau} \tilde{\mathcal B}(t,s) \,ds \bigg)^{\frac{p-q}{pq}} \bigg) \\
	& \hspace{-2.5cm} + \sup_{t \in (0,\infty)} \frac{t^{\frac{1}{p}}}{v_1(t)^{\frac{1}{m}}}  \bigg( \int_t^{\infty} \bigg[\sup_{\tau \in [x,\infty)}\bigg[\sup_{y \in [\tau,\infty)} \bigg(\frac{1}{\phi(y)}\bigg)^\frac{pq}{p-q}\bigg]\bigg(\int_{t}^{\tau} \tilde{\mathcal B}(t,s) \,ds \bigg) \bigg]\bigg(\int_{t}^{x} w(y)\,dy \bigg)^{\frac{q}{p-q}} w(x) \,dx \bigg)^{\frac{p-q}{pq}}
	\\
	& \hspace{-2.5cm} + \sup_{t \in (0,\infty)} \frac{t^{\frac{1}{p}}}{v_1(t)^{\frac{1}{m}}}  \bigg( \int_t^{\infty}\bigg(\int_{x}^{\infty} \bigg(\sup_{\tau \in [y,\infty)} \frac{B(\tau)^{\frac{1}{\alpha}}}{\phi(\tau)} \bigg)^q w(y)\, dy\bigg)^{\frac{q}{p-q}} \bigg( \sup_{\tau \in [x,\infty)} \frac{B(\tau)^{\frac{1}{\alpha}}}{\phi(\tau)} \bigg)^q \bigg(\int_{t}^{x} \tilde{\mathcal B}(t,s) \,ds \bigg) w(x) \,dx\bigg)^{\frac{p-q}{pq}},
	\end{align*} 
	where 
	$$
	\tilde{\mathcal B}(t,s) := \bigg( \int_t^s \bigg( \frac{B(y)}{y} \bigg)^{\frac{p}{p - \alpha}}\,dy \bigg)^{\frac{p(q - \alpha)}{\alpha (p - q)}}  \,  \bigg( \frac{B(s)}{s} \bigg)^{\frac{p}{p - \alpha}}, \qquad 0 < t < s < \infty.
	$$
\end{itemize}
\end{theorem}	

\begin{theorem}\label{main3stresult}
Let $0 < p \le \alpha \le r < \infty$, $\alpha < \min\{m,\,q\} < \infty$ and $b \in \W (0,\infty) \cap \mp^+ ((0,\infty);\dn)$ be such that the function $B(t)$ satisfies  $0 < B(t) < \infty$ for every $t \in (0,\infty)$, $B \in \Delta_2$, $B(\infty) = \infty$ and $B(t) / t^{\alpha / r}$ is quasi-increasing. Moreover, let $\phi \in \W (0,\infty) \cap C(0,\infty)$ be such that $\phi \in Q_{r}(0,\infty)$ is a quasi-increasing function. Assume that $v \in \W_{m,p}(0,\infty)$ and $w \in \W(0,\infty)$.
\begin{itemize}
	\item[i)] If $m \le q$, then
	\begin{align*}
	\|M_{\phi,\Lambda^{\alpha}(b)}\|_{\GG(p,m,v) \rightarrow \Lambda^q (w)} & \\
	& \hspace{-3cm} \approx \, \sup_{t \in (0,\infty)} \bigg( \int_t^{\infty} \frac{v_0(s)}{v_1(s)^{\frac{2m - \alpha}{m - \alpha}}}\,ds \bigg)^{\frac{m - \alpha}{m \alpha}}  \bigg( \int_0^{t} \bigg( \sup_{\tau \in [x,t]} \frac{B(\tau)^{\frac{1}{\alpha}}}{\phi(\tau)}\bigg)^q \,w(x)\,dx\bigg)^{\frac{1}{q}} 
	\\
	& \hspace{-2.5cm} + \sup_{t \in (0,\infty)} \bigg( \int_0^t \frac{B(s)^{\frac{m}{m - \alpha}}v_0(s)}{v_1(s)^{\frac{2m - \alpha}{m - \alpha}}}\,ds \bigg)^{\frac{m - \alpha}{m \alpha}} \bigg( \sup_{\tau \in[t,\infty)} \frac{1}{\phi (\tau)} \bigg) \bigg( \int_0^{t} w(x)\,dx\bigg)^{\frac{1}{q}}
	\\
	& \hspace{-2.5cm} + \sup_{t \in (0,\infty)} \bigg( \int_0^t \frac{B(s)^{\frac{m}{m - \alpha}}v_0(s)}{v_1(s)^{\frac{2m - \alpha}{m - \alpha}}}\,ds \bigg)^{\frac{m - \alpha}{m \alpha}} \bigg( \int_t^{\infty} \bigg( \sup_{\tau \in[x,\infty)}  \frac{1}{\phi (\tau)} \bigg)^q  w(x)\,dx\bigg)^{\frac{1}{q}};
	\end{align*}
		
		\item[ii)] If $q < m$, then
		\begin{align*}
		\|M_{\phi,\Lambda^{\alpha}(b)}\|_{\GG(p,m,v) \rightarrow \Lambda^q (w)} & \\
		& \hspace{-3cm} \approx \, \bigg(\int_{0}^{\infty}\bigg[\sup_{\tau \in [t,\infty)} \bigg( \frac{B(\tau)^{\frac{1}{\alpha}}}{\phi(\tau)} \bigg)^\frac{mq}{m-q} \bigg(\int_{\tau}^{\infty} \tilde{\mathfrak{B}}_1 (s) \,ds \bigg)\bigg]\bigg(\int_{0}^{t} w(x)\,dx\bigg)^\frac{q}{m-q}w(t) \, dt\bigg)^\frac{m-q}{mq}
		\\
		& \hspace{-2.5cm} + \bigg(\int_{0}^{\infty}\bigg(\int_{0}^{t}\bigg( \sup_{y \in [x, t]} \frac{B(y)^{\frac{1}{\alpha}}}{\phi(y)} \bigg)^q \,w(x)\,dx\bigg)^\frac{q}{m-q} \bigg[\sup_{\tau \in [t,\infty)} \bigg( \frac{B(\tau)^{\frac{1}{\alpha}}}{\phi(\tau)} \bigg)^q \bigg(\int_{\tau}^{\infty} \tilde{\mathfrak{B}}_1 (s) \,ds \bigg)\bigg]w(t) \, dt\bigg)^\frac{m-q}{mq}
		\\
		& \hspace{-2.5cm} + \bigg(\int_{0}^{\infty}\bigg[\sup_{\tau \in [t,\infty)}\bigg[\sup_{s \in [\tau,\infty)}\bigg(\frac{1}{\phi(s)}\bigg)^{\frac{mq}{m-q}} \bigg]\bigg(\int_{0}^{\tau} \tilde{\mathfrak{B}}_2 (s) \,ds \bigg)\bigg] \bigg(\int_{0}^{t}w(x)\,dx\bigg)^\frac{q}{m-q} w(t) \, dt\bigg)^\frac{m-q}{mq}
		\\
		& \hspace{-2.5cm} + \bigg(\int_{0}^{\infty}\bigg(\int_{t}^{\infty}\bigg( \sup_{\tau \in [x,\infty) } \frac{1}{\phi(\tau)}\bigg)^q w(x)\,dx\bigg)^\frac{q}{m-q} \, \bigg( \sup_{\tau \in [t,\infty)} \frac{1}{\phi(\tau)}\bigg)^q \bigg(\int_{0}^{t} \tilde{\mathfrak{B}}_2 (s) \,ds \bigg) w(t) \,dt\bigg)^\frac{m-q}{mq},
		\end{align*}  
		where functions $\tilde{\mathfrak{B}}_1$ and $\tilde{\mathfrak{B}}_2$ are defined for all $s \in (0,\infty)$ by
		$$
		\tilde{\mathfrak{B}}_1(s) : = \bigg( \int_s^{\infty} \frac{v_0(t)}{v_1(t)^{\frac{2m - \alpha}{m - \alpha}}}\,dt\bigg)^{\frac{m(q - \alpha)}{\alpha (m - q)}} \frac{v_0(s)}{v_1(s)^{\frac{2m - \alpha}{m - \alpha}}}, \quad 
		\tilde{\mathfrak{B}}_2(s) : = \bigg( \int_0^s \frac{B(t)^{\frac{m}{m - \alpha}}v_0(t)}{v_1(t)^{\frac{2m - \alpha}{m - \alpha}}}\,dt \bigg)^{\frac{m(q - \alpha)}{\alpha (m - q)}} \frac{B(s)^{\frac{m}{m - \alpha}}v_0(s)}{v_1(s)^{\frac{2m - \alpha}{m - \alpha}}},
		$$	
		respectively.
	\end{itemize}
\end{theorem}

\begin{theorem}\label{main4stresult}
	Let $0 < \alpha \le r < \infty$, $\alpha < \min\{m,\,p,\,q\} < \infty$ and  and $b \in \W (0,\infty) \cap \mp^+ ((0,\infty);\dn)$ be such that the function $B(t)$ satisfies  $0 < B(t) < \infty$ for every $t \in (0,\infty)$, $B \in \Delta_2$, $B(\infty) = \infty$ and $B(t) / t^{\alpha / r}$ is quasi-increasing. Moreover, let $\phi \in \W (0,\infty) \cap C(0,\infty)$ be such that $\phi \in Q_{r}(0,\infty)$ is a quasi-increasing function. Assume that $v \in \W_{m,p}(0,\infty)$ and $w \in \W(0,\infty)$.
	Suppose that
	\begin{gather*}
	\int_0^t \tilde{v}_2(s)\,ds < \infty, \quad \int_t^{\infty} s^{-\frac{m \alpha}{p (m - \alpha)}} \tilde{v}_2(s)\,ds < \infty, \qquad 0 < \int_0^t \bigg( \int_s^t \bigg( \frac{B(y)}{y} \bigg)^{\frac{p}{p - \alpha}}\,dy \bigg)^{\frac{m (p - \alpha)}{p (m - \alpha)}} \tilde{v}_2(s)\,ds < \infty, \qquad t \in (0,\infty), \\
    \int_0^1 s^{-\frac{m \alpha}{p (m - \alpha)}} \tilde{v}_2(s)\,ds = \int_1^{\infty} \tilde{v}_2(s) \,ds = \infty,
	\end{gather*}
	where the function $\tilde{v}_2$ is defined by
	\begin{equation*}
	\tilde{v}_2(t) : = \frac{t^{\frac{m (p - \alpha)}{p (m - \alpha)}} v_0(t)}{v_1(t)^{\frac{2m - \alpha}{m - \alpha}}}, \qquad t \in (0,\infty).
	\end{equation*} 
	
	\item[i)] If  $\max \{p,\,m\} \le q$, then
	\begin{align*}
	\|M_{\phi,\Lambda^{\alpha}(b)}\|_{\GG(p,m,v) \rightarrow \Lambda^q (w)} & \\
	& \hspace{-3cm} \approx \, \sup_{t \in (0,\infty)}\bigg( \int_0^t \bigg( \int_s^t \bigg( \frac{B(y)}{y} \bigg)^{\frac{p}{p - \alpha}} \, dy \bigg)^{\frac{m (p - \alpha)}{p (m - \alpha)}} \tilde{v}_2(s) \, ds \bigg)^{\frac{m - \alpha}{m \alpha}} 
	\bigg( \sup_{\tau \in [t,\infty)} \frac{1}{\phi(\tau)} \bigg) \bigg( \int_0^{t} w(x) \, dx 	\bigg)^{\frac{1}{q}}
	\\
	& \hspace{-2.5cm} + \sup_{t \in (0,\infty)}\bigg( \int_0^t \bigg( \int_s^t \bigg( \frac{B(y)}{y} \bigg)^{\frac{p}{p - \alpha}} \, dy \bigg)^{\frac{m (p - \alpha)}{p (m - \alpha)}} \tilde{v}_2(s) \, ds \bigg)^{\frac{m - \alpha}{m \alpha}} 
	\bigg(\int_t^{\infty} \bigg( \sup_{\tau \in [x,\infty)} \frac{1}{\phi(\tau)} \bigg)^q w(x) \,dx \bigg)^{\frac{1}{q}}
	\\
	& \hspace{-2.5cm} + \sup_{t \in (0,\infty)}  \bigg( \int_0^{\infty} \bigg( \frac{1}{s + t}\bigg)^{\frac{m \alpha}{p (m - \alpha)}} \tilde{v}_2(s) \, ds \bigg)^{\frac{m - \alpha}{m \alpha}} 
	\bigg(\int_0^{t} \bigg( \sup_{\tau \in [x,t]} \frac{B(\tau)^{\frac{1}{\alpha}}}{\phi(\tau)}\bigg)^q w(x) \,dx \bigg)^{\frac{1}{q}};
	\end{align*}
	
	\item[ii)] If $m \le q < p$, then
	\begin{align*}
	\|M_{\phi,\Lambda^{\alpha}(b)}\|_{\GG(p,m,v) \rightarrow \Lambda^q (w)} & \\
	& \hspace{-3cm} \approx \, \sup_{t \in (0,\infty)}\bigg( \int_0^t \bigg( \int_s^t \bigg( \frac{B(y)}{y} \bigg)^{\frac{p}{p - \alpha}} \, dy \bigg)^{\frac{m (p - \alpha)}{p (m - \alpha)}} \tilde{v}_2(s) \, ds \bigg)^{\frac{m - \alpha}{m \alpha}} 
	\bigg( \sup_{\tau \in [t,\infty)} \frac{1}{\phi(\tau)} \bigg) \bigg( \int_0^{t} w(x) \, dx 	\bigg)^{\frac{1}{q}}
	\\
	& \hspace{-2.5cm} + \sup_{t \in (0,\infty)}\bigg( \int_0^t \bigg( \int_s^t \bigg( \frac{B(y)}{y} \bigg)^{\frac{p}{p - \alpha}} \, dy \bigg)^{\frac{m (p - \alpha)}{p (m - \alpha)}} \tilde{v}_2(s) \, ds \bigg)^{\frac{m - \alpha}{m \alpha}} 
	\bigg(\int_t^{\infty} \bigg( \sup_{\tau \in [x,\infty)} \frac{1}{\phi(\tau)} \bigg)^q w(x) \,dx \bigg)^{\frac{1}{q}}
	\\
	& \hspace{-2.5cm} + \sup_{t \in (0,\infty)}\bigg( \int_0^t \tilde{v}_2 (s)\,ds \bigg)^{\frac{m - \alpha}{m \alpha}} \bigg( \sup_{\tau \in [t,\infty)} \frac{1}{\phi(\tau)} \bigg( \int_{t}^{\tau} \tilde{\mathcal B}(t,y) \,dy \bigg)^{\frac{p-q}{pq}} \bigg)
	\bigg(\int_0^{t} w(x) \, dx\bigg)^{\frac{1}{q}} \\
	& \hspace{-2.5cm} + \sup_{t \in (0,\infty)} \bigg( \int_0^t \tilde{v}_2 (s)\,ds \bigg)^{\frac{m - \alpha}{m \alpha}} \times \\
	& \hspace{-0.5cm} \times \bigg( \int_t^{\infty} \bigg[\sup_{\tau \in [s,\infty)}\bigg[\sup_{x \in [\tau,\infty)} \bigg(\frac{1}{\phi(x)}\bigg)^\frac{pq}{p-q}\bigg]\bigg(\int_{t}^{\tau} \tilde{\mathcal{B}}(t,y) \,dy\bigg)\bigg]\bigg(\int_{t}^{s}w(y)\,dy \bigg)^{\frac{q}{p-q}} w(s) \,ds\bigg)^{\frac{p-q}{pq}}
	\\
	& \hspace{-2.5cm} + \sup_{t \in (0,\infty)}\bigg( \int_0^t \tilde{v}_2 (s)\,ds \bigg)^{\frac{m - \alpha}{m \alpha}} \times \\
	& \hspace{-0.5cm} \times \bigg( \int_t^{\infty}\bigg(\int_{s}^{\infty} \bigg( \sup_{\tau \in [x,\infty)} \frac{1}{\phi(\tau)}\bigg)^q w(x) \, dx\bigg)^{\frac{q}{p-q}} \bigg( \sup_{\tau \in [s,\infty)} \frac{1}{\phi(\tau)}\bigg)^q \bigg(\int_{t}^{s} \tilde{\mathcal{B}}(t,y) \,dy \bigg) w(s) \,ds\bigg)^{\frac{p-q}{pq}}
	\\
	& \hspace{-2.5cm} + \sup_{t \in (0,\infty)} \bigg(\int_{0}^{\infty} \bigg(\frac{t}{s + t}\bigg)^{\frac{m \alpha}{p (m - \alpha)}} \tilde{v}_2(s) \, ds\bigg)^{\frac{m - \alpha}{m \alpha}} \times \\
	& \hspace{-0.5cm} \times \bigg(\int_{0}^{\infty}\bigg[\sup_{\tau \in [t,\infty)} \bigg( \frac{B(\tau)^{\frac{1}{\alpha}}}{\phi(\tau)} \bigg)^\frac{pq}{p-q}(\tau+t)^\frac{q}{q-p}\bigg]\bigg(\int_{0}^{t}w(x)\,dx\bigg)^\frac{q}{p-q} w(t)\, dt\bigg)^\frac{p-q}{pq}
	\\
	& \hspace{-2.5cm} + \sup_{t \in (0,\infty)} \bigg(\int_{0}^{\infty} \bigg(\frac{t}{s + t}\bigg)^{\frac{m \alpha}{p (m - \alpha)}} \tilde{v}_2(s) \, ds\bigg)^{\frac{m - \alpha}{m \alpha}} \times \\
	& \hspace{-0.5cm} \times \bigg(\int_{0}^{\infty}\bigg(\int_{0}^{t}\bigg( \sup_{y \in [x, t]} \frac{B(y)^{\frac{1}{\alpha}}}{\phi(y)} \bigg)^q \,w(x)\,dx\bigg)^\frac{q}{p-q} \bigg[\sup_{\tau \in [t,\infty)}\bigg( \frac{B(\tau)^{\frac{1}{\alpha}}}{\phi(\tau)} \bigg)^q (\tau+t)^\frac{q}{q-p}\bigg] w(t) \, dt\bigg)^\frac{p-q}{pq}.
	\end{align*}
\end{theorem}


\end{document}